\documentclass[final, reqno, 11pt]{amsart}
\usepackage{amsmath}
\usepackage{amsfonts}
\usepackage{amssymb}
\usepackage{amsthm}

\newtheorem{theorem}{Theorem}[section]
\newtheorem{lemma}[theorem]{Lemma}
\newtheorem{proposition}[theorem]{Proposition}

\theoremstyle{remark}
\newtheorem{observation}[theorem]{Remark}

\newtheorem{definition}{Definition}

\newcommand{\set}{\mathbb}
\newcommand{\dl}{\nabla}
\newcommand{\les}{\lesssim}

\newcommand{\mc}{\mathcal}
\newcommand{\be}{\begin{equation}}
\newcommand{\ee}{\end{equation}}
\newcommand{\bee}{\begin{align}}
\newcommand{\eee}{\end{align}}
\newcommand{\ba}{\begin{array}}
\newcommand{\ds}{\displaystyle}

\newcommand{\ea}{\end{array}}
\newcommand{\bpm}{\begin{pmatrix}}
\newcommand{\epm}{\end{pmatrix}}
\newcommand{\lb}{\label}

\DeclareMathOperator{\sgn}{sgn}
\DeclareMathOperator{\supp}{supp}

\DeclareMathOperator{\Imim}{Im}

\DeclareMathOperator{\B}{\mathcal B}

\DeclareMathOperator*{\esssup}{ess\; sup}
\newcommand{\ov}{\overline}
\newcommand{\dd}{{\,}{d}}

\renewcommand{\Im}{\Imim}
\newcommand{\R}{\mathbb R}
\newcommand{\C}{\mathbb C}

\newcommand{\V}{\mathcal V}
\newcommand{\W}{\mathcal W}

\title[Centre-stable manifold]{A Centre-Stable Manifold for the Energy-Critical Wave Equation in $\R^3$ in the Symmetric Setting}
\author{Marius Beceanu}
\address{Rutgers University Department of Mathematics, 110 Frelinghuysen Rd., Piscataway, NJ, 08854, USA}
\email{mbeceanu@math.rutgers.edu}
\thanks{The author was partially supported by a Rutgers Research Council grant.}
\begin{document}
\maketitle
\numberwithin{equation}{section}
\begin{abstract} Consider the focusing semilinear wave equation in $\R^3$ with energy-critical nonlinearity
$$
\partial_t^2 \psi - \Delta \psi - \psi^5 = 0,\ \psi(0) = \psi_0,\ \partial_t \psi(0) = \psi_1.
$$
This equation admits stationary solutions of the form
$$
\phi(x, a) := (3a)^{1/4} (1+a|x|^2)^{-1/2},
$$
called solitons, which solve the elliptic equation
$$
-\Delta \phi - \phi^5 = 0.
$$
Restricting ourselves to the space of symmetric solutions $\psi$ for which $\psi(x) = \psi(-x)$, we find a local centre-stable manifold, in a neighborhood of $\phi(x, 1)$, for this wave equation in the weighted Sobolev space $\langle x \rangle^{-1} \dot H^1 \times \langle x \rangle^{-1} L^2$. Solutions with initial data on the manifold exist globally in time for $t \geq 0$, depend continuously on initial data, preserve energy, and can be written as the sum of a rescaled soliton and a dispersive radiation term.

The proof is based on a new class of reverse Strichartz estimates, introduced in \cite{becgol} and adapted here to the case of Hamiltonians with a resonance.
\end{abstract}

\tableofcontents
\section{Introduction}
\subsection{Overview}
In this paper we study the semilinear wave equation
\be\lb{wave}
\partial_t^2 \psi - \Delta \psi - \psi^5 = 0,\ \psi(0) = \psi_0,\ \partial_t \psi(0) = \psi_1.
\ee
Energy $E(t)$ is an invariant quantity for equation (\ref{wave}). $E(t)$ is given~by
$$
E(t) := \frac 1 2 \int_{\R^3} |\dl \psi(t)|^2 + (\partial_t \psi(t))^2 \dd x - \frac 1 6 \int_{\R^3} (\psi(t))^6 \dd x.
$$
Equation (\ref{wave}) also admits special stationary solutions of the form $\psi(x, t) = \phi(x)$, where $\phi$ is a positive solution of the semilinear elliptic equation
$$
-\Delta \phi - \phi^5 = 0.
$$
Such solutions exist and are unique up to translation and rescaling, see \cite{aub}, being explicitly given by
$$
\phi(x, a) := (3a)^{1/4} (1+a|x|^2)^{-1/2} = a^{1/4} \phi(a^{1/2}x, 1).
$$
Observe that $\phi(x, a) \in \dot H^1$ with constant norm. In what follows we study small stable perturbations of $\phi(x, a)$, i.e. solutions to (\ref{wave}) that stay close to $\phi(x, a)$ in the $\dot H^1$ norm for all times.

We first examine the spectral properties of $H(a)$. It is known that the continuous spectrum of $H(a)$ is $[0, \infty)$, while the point spectrum contains one negative eigenvalue $-k(a)^2$, $H(a) g(x,a) = -k(a)^2 g(x,a)$, and three eigenvectors at zero, $\partial_{x_j} \phi(x, a)$, $1 \leq j \leq 3$. Here $g(x, a)$ is bounded, radially symmetric and exponentially decaying.

Note also the presence of a zero resonance, $\partial_a \phi(x, a)$ --- a bounded function that satisfies the equation $(-\Delta+V)f=0$ without belonging to $L^2$.

Due to scaling, $k(a) = a^{1/2} k(1)$, we can set $g(x,a) = g(a^{1/2}x, 1)$, and $\partial_a \phi(x, a) = a^{-3/4} \partial_a \phi(a^{1/2}x, 1)$, $\partial_{x_j} \phi(x, a) = a^{3/4} \partial_{x_j} \phi(a^{1/2} x, 1)$.

In the following we set $\phi:=\phi(1)$, $k:=k(1)$, $g:=g(1)$, $V:=V(1)$, $H:=H(1)$, etc.

Furthermore,
$$
\langle \partial_a \phi, g \rangle = k^{-1} \langle \partial_a \phi, Hg \rangle = k^{-1} \langle H \partial_a \phi, g \rangle = 0.
$$

Henceforth we restrict ourselves to the subspace of symmetric functions, i.e.\ functions $f$ for which $f(x)=f(-x)$ for all $x \in \R^3$. The restriction of $H(a)$ to this subspace keeps only the continuous spectrum, the radial negative eigenvalue and the radial resonance at zero.

We denote Lorentz spaces by $L^{p, q}$. For their definition and properties see \cite{bergh}. Note that $L^{3/2, 1}$ is the dual of weak-$L^3$ and that $|\dl|^{-1} L^{3/2, 1} \subset L^{3, 1}$.

Our first result is the following:
\begin{theorem}\lb{main_theorem} For some small $\epsilon>0$, consider the set
$$\begin{aligned}
\mc N_0 = \{&(\psi_0, \psi_1) \in \dot H^1 \times L^2 \mid \|\psi_0 - \phi\|_{\dot H^1 \cap |\dl|^{-1} L^{3/2, 1}} + \|\psi_1\|_{L^2 \cap L^{3/2, 1}} < \epsilon,\\
&\langle k(\psi_0 - \phi) -\psi_1, g \rangle = 0,\ \psi_0(x) = \psi_0(-x),\ \psi_1(x) = \psi_1(-x)\}.
\end{aligned}$$
There exists a function $h:\mc N_0 \to \R$ such that for any $(\psi_0, \psi_1) \in \mc N_0$ the equation (\ref{wave}) with initial data $(\psi_0+h(\psi_0, \psi_1)g, \psi_1 - h(\psi_0, \psi_1)kg)$ has a global solution $\psi(\psi_0, \psi_1) = u + \phi(a(t))$ such that $a(0)=1$, $\|\dot a(t)\|_{L^1 \cap L^\infty} \les \|\psi_0 - \phi\|_{\dot H^1 \cap |\dl|^{-1} L^{3/2, 1}} + \|\psi_1\|_{L^2 \cap L^{3/2, 1}}$ and for $t \geq 0$
$$
\|u(t)\|_{L^{6, 2}_x L^\infty_t \cap L^\infty_x L^2_t \cap L^8_{x, t} \cap L^\infty_x L^1_t} \les \|\psi_0 - \phi\|_{\dot H^1 \cap |\dl|^{-1} L^{3/2, 1}} + \|\psi_1\|_{L^2 \cap L^{3/2, 1}}.
$$
Energy remains bounded:
$$
\|\psi(\psi_0, \psi_1)(t) - \phi\|_{\dot H^1} + \|\partial_t \psi(\psi_0, \psi_1)(t)\|_{L^2} \les \|\psi_0 - \phi\|_{\dot H^1 \cap |\dl|^{-1} L^{3/2, 1}} + \|\psi_1\|_{L^2 \cap L^{3/2, 1}}.
$$
$h(\psi_0, \psi_1)$ and $\psi(\psi_0, \psi_1)$ have the following further properties:
$$\begin{aligned}
|h(\psi_0, \psi_1)| &\les (\|\psi_0 - \phi\|_{\dot H^1 \cap |\dl|^{-1} L^{3/2, 1}} + \|\psi_1\|_{L^2 \cap L^{3/2, 1}})^2;\\
|h(\psi_0^1, \psi_1^1) - h(\psi_0^2, \psi_1^2)| &\les \epsilon(\|\psi_0^1 - \psi_0^2\|_{\dot H^1 \cap |\dl|^{-1} L^{3/2, 1}} + \|\psi_1^1 - \psi_1^2\|_{L^2 \cap L^{3/2, 1}});\\
\|\dot a(\psi_0^1, \psi_1^1) - \dot a(\psi_0^2, \psi_1^2)\|_{L^1_t} &\les \|\psi_0^1 - \psi_0^2\|_{\dot H^1 \cap |\dl|^{-1} L^{3/2, 1}} + \|\psi_1^1 - \psi_1^2\|_{L^2 \cap L^{3/2, 1}};\\
\|u(\psi_0^1, \psi_1^1) - u(\psi_0^2, \psi_1^2)\|_{L^{6, 2}_x L^\infty_t \cap L^\infty_x L^2_t \cap L^\infty_x L^1_t} &\les \|\psi_0^1 - \psi_0^2\|_{\dot H^1 \cap |\dl|^{-1} L^{3/2, 1}} + \|\psi_1^1 - \psi_1^2\|_{L^2 \cap L^{3/2, 1}};\\
\|\psi(\psi_0^1, \psi_1^1) - \psi(\psi_0^2, \psi_1^2)\|_{L^{6, 2}_x L^\infty_t} &\les \|\psi_0^1 - \psi_0^2\|_{\dot H^1 \cap |\dl|^{-1} L^{3/2, 1}} + \|\psi_1^1 - \psi_1^2\|_{L^2 \cap L^{3/2, 1}}.
\end{aligned}$$
\end{theorem}
In other words, we have a codimension-one Lipschitz manifold of initial data in $\dot H^1 \cap |\dl|^{-1} L^{3/2, 1} \times L^2 \cap L^{3/2, 1}$
$$
\mc N = \{(\tilde \psi_0 = \psi_0+h(\psi_0, \psi_1)g, \tilde \psi_1 = \psi_1 - h(\psi_0, \psi_1)kg) \mid (\psi_0, \psi_1) \in \mc N_0\},
$$
for which the solution to (\ref{wave}) exists globally in time, depends continuously on the initial data, and can be written as the sum of a rescaled soliton and a dispersive term or of a fixed soliton and a term that preserves energy.

Note that the setting in which the energy is preserved is not the same as the setting in which the radiation term disperses.

By rescaling we obtain similar manifolds in the neighborhood of each soliton $\phi(x, a)$. These manifolds never intersect, due to the norms in which they are defined: $\phi(x, a_1) - \phi(x, a_2) \not\in L^{3, 1}$ for $a_1 \ne a_2$, while $|\dl|^{-1} L^{3/2, 1} \subset~L^{3/2, 1}$.

The proof of Theorem \ref{main_theorem} is based on the subsequent reverse Strichartz estimates of Proposition \ref{prop28}, as well as on the method of modulation (\cite{soffer1}, \cite{soffer2}). We use modulation to handle the resonance at zero and the similar method introduced in \cite{sch} and \cite{krisch} to control the projection on the imaginary spectrum, see below.

\begin{observation} In the conditions on initial data, $L^{3/2, 1}$ can be replaced by $\mc K_0$, the Kato space closure of the set of bounded, compactly supported functions, where the Kato space $\mc K$ is defined by
$$
\mc K = \Big\{f \mid \sup_y \int_{\R^3} \frac {|f(x)|}{|x-y|} \dd x < \infty\Big\}.
$$
\end{observation}

The manifold $\mc N$ also enjoys a uniqueness property formulated in the following statement.

\begin{definition}
We call a solution $\psi(t)$ of (\ref{wave}) \emph{small orbitally stable} if, for some small fixed $\epsilon>0$ and for all $t \geq 0$ $\psi(t) = \phi + v(t)$, $\|(v(0), \partial_t v(0))\|_{|\dl|^{-1} L^{3/2, 1} \cap \dot H^1 \times L^{3/2, 1} \cap L^2} < \epsilon$, $\|v(t)\|_{L^{6, 2}_x L^\infty_t} < \epsilon$, and $\|\partial_t v(t)\|_{L^\infty_t L^2_x} < \infty$.
\end{definition}
\begin{proposition}\lb{sas} If $\psi(t)$ is a small orbitally stable solution, then $\psi(0) \in \mc N$.
\end{proposition}

In order to state our next result, we introduce the notion of a \emph{centre-stable manifold}, following Bates--Jones \cite{bates}. These authors proved for a large class of semilinear equations that the space of solutions locally decomposes into an unstable and a centre-stable manifold. Their result is as follows. Consider a Banach space $X$ and the semilinear equation
\be\lb{117}
u_t = A u + f(u),
\ee
under the assumptions
\begin{enumerate}
\item[H1] $A:X \to X$ is a closed, densely defined linear operator that generates a $C_0$ group.
\item[H2]\lb{H2} The spectrum of $A$, $\sigma(A) = \sigma_s(A) \cup \sigma_c(A) \cup \sigma_u(A)$, decomposes into left half-plane (stable), imaginary (centre), and right half-plane (unstable) components. The stable and unstable components, $\sigma_s(A)$ and $\sigma_u(A)$, are bounded.
\item[H3]\lb{H3} The nonlinearity $f$ is locally Lipschitz, $f(0) = 0$, and for every $\epsilon>0$ there exists a neighborhood of $0 \in X$ on which $\|f(x) - f(y)\| \leq \epsilon \|x-y\|$.
\end{enumerate}
Let $X^u$, $X^c$, and $X^s$ be the $A$-invariant subspaces corresponding to $\sigma_u$, $\sigma_c$, and $\sigma_s$ and let $S^c(t)$ be the evolution generated by $A$ on $X^c$. \cite{bates} further assume that
\begin{enumerate}
\item[C1-2] $\dim X^u$ and $\dim X^s$ are finite.
\item[C3]\lb{C3} $S^c$ has subexponential growth: $\forall \rho>0$ $\exists M>0$ such that $\|S^c(t)\| \leq M e^{\rho|t|}$.
\end{enumerate}

Let $\Upsilon$ be the flow on $X$ generated by (\ref{117}). $\mc N \subset U$ is called $t$-invariant if $\Upsilon(s)v \in U$ for all $s \in [0, t]$ implies that $\Upsilon(s)v \in \mc N$ for $s \in [0, t]$.

\begin{definition}\lb{unstable}
Let the \emph{unstable} manifold $W^u \subset U$ be the set of solutions that remain in $U$ for all $t<0$ and decay exponentially as $t \to -\infty$:
$$
W^u = \{u \in U \mid \forall t \leq 0\ \Upsilon(t) u \in U,\ \exists C_1>0\ \forall t \leq 0\ \|\Upsilon(t) u\|_X \les e^{C_1 t} \}.
$$
\end{definition}
Also consider the canonical direct sum spectral projection $\pi^{cs}$ onto the centre-stable part of the spectrum: $\pi^{cs}(X) = X^c \oplus X^s$.
\begin{definition}
A \emph{centre-stable} manifold $\mc N \subset U$ is a Lipschitz manifold (i.e.\ parametrized by Lipschitz maps) such that $\mc N$ is $t$-invariant relative to $U$, $\pi^{cs}(\mc N)$ contains a neighborhood of $0$ in $X^c \oplus X^s$, and $\mc N \cap W^u = \{0\}$.
\lb{centr}
\end{definition}

The conclusion of \cite{bates} is then
\begin{theorem}
Under assumptions H1-H3 and C1-C3, locally around $0$, there exist an unstable Lipschitz manifold $W^u$ tangent to $X^u$ at $0$ and a centre-stable manifold $W^{cs}$ tangent to $X^{cs}$ at $0$.
\lb{t1}
\end{theorem}

Consider the set
$$\begin{aligned}
\tilde {\mc N_0} = \{&(\psi_0, \psi_1) \in \langle x \rangle^{-1} \dot H^1 \times \langle x \rangle^{-1} L^2 \mid \|\psi_0 - \phi\|_{\langle x \rangle^{-1} \dot H^1} + \|\psi_1\|_{\langle x \rangle^{-1} L^2} < \epsilon,\\
&\langle k(\psi_0 - \phi) -\psi_1, g \rangle = 0,\ \psi_0(x) = \psi_0(-x),\ \psi_1(x) = \psi_1(-x)\}.
\end{aligned}$$
Note that $\langle x \rangle^{-1} \dot H^1 \times \langle x \rangle^{-1} L^2 \subset |\dl|^{-1} L^{3/2, 1} \cap \dot H^1 \times L^{3/2, 1} \cap L^2$. Then
$$
\tilde{\mc N} = \{(\tilde \psi_0 = \psi_0 + h(\psi_0, \psi_1) g, \tilde \psi_1 = \psi_1 - h(\psi_0, \psi_1) k g) \mid (\psi_0, \psi_1) \in \tilde {\mc N_0}\}
$$
is a set of initial data which lead to global-in-time solutions for $t \geq 0$.
\begin{proposition}[Main result]\lb{invariant} Consider a solution $\psi(\psi_0, \psi_1)$ to (\ref{wave}) with initial data
$$
(\tilde \psi_0 = \psi_0 + h(\psi_0, \psi_1) g(a), \tilde \psi_1 = \psi_1 - h(\psi_0, \psi_1) k g) \in \tilde {\mc N},
$$
where $(\psi_0, \psi_1) \in \tilde {\mc N_0}$. Then $(\psi(\psi_0, \psi_1)(t), \partial_t \psi(\psi_0, \psi_1)(t)) \in \tilde{\mc N}$ for sufficiently small $|t|$ (in other words, $\tilde {\mc N}$ is locally in time invariant under the action of the equation (\ref{wave})) and $(\psi(\psi_0, \psi_1)(t), \partial_t \psi(\psi_0, \psi_1)(t)) \in \langle x \rangle^{-1} \dot H^1 \times \langle x \rangle^{-1} L^2$ for all $t \geq 0$.

Furthermore, $\tilde{\mc N}$ is a centre-stable manifold for (\ref{wave}).
\end{proposition}

The set $\tilde {\mc N}$ is not optimal from the point of view of scaling (one only needs half a power of decay instead of a full power), but we choose this setting to simplify the computations.

\subsection{Linear estimates} Consider a Hamiltonian of the form $H=-\Delta+V$ with a resonance at zero, where $V$ is a real-valued scalar potential on~$\R^3$.

We assume $\langle x \rangle V \in L^{3/2, 1}$. By \cite{simon} this is sufficient to guarantee the self-adjointness of $H=-\Delta+V$. It has been shown in \cite{bec} that $H$ has only finitely many negative eigenvalues.




We prove dispersive estimates for the continuous part of the spectrum, in which we have to account for the resonance as well. For simplicity, we assume in the course of the proof that $H$ has only one negative eigenvalue $-k^2$ with corresponding eigenfunction $g$, but the proof works in the same manner for any finite number of negative eigenvalues.

Our starting point is Lemma \ref{lemma24}, which provides an expansion of the Kato-Birman operator $(I+VR_0(\lambda^2))^{-1}$ at zero. We also make use of Lemma \ref{lemma_free}, which treats the case of the free evolution. This leads to the reverse Strichartz estimates of Proposition \ref{prop28}:
\begin{proposition}\lb{prop28} Assume that $\langle x \rangle V \in L^{3/2, 1}$ and that $H=-\Delta+V$ has a resonance $\phi$ at zero. Then for $t \geq 0$
$$\begin{aligned}
\frac {\sin(t \sqrt H) P_c}{\sqrt H} f(x) &= -\frac {4\pi}{\langle V, \phi\rangle^2} \phi \otimes V\phi \int_0^t\frac {\sin(s \sqrt {-\Delta})}{\sqrt {-\Delta}} f(x) \dd s + S(t) f(x)\\
\cos(t\sqrt H) P_c f(x) &= -\frac {4\pi}{\langle V, \phi\rangle^2} \phi \otimes V\phi \int_0^t \cos(s\sqrt{-\Delta}) f(x) \dd s + C(t) f(x),
\end{aligned}$$
where $S(t)$ and $C(t)$ satisfy reverse Strichartz estimates:
$$\begin{aligned}
\|S(t) f\|_{L^{6, 2}_x L^\infty_t \cap L^\infty_x L^2_t} &\les \|f\|_2;\\
\|S(t)f\|_{L^\infty_x L^1_t} &\les \|f\|_{L^{3/2, 1}};\\
\Big\|\int_0^t S(t-s) F(s) \dd s\Big\|_{L^{6, 2}_x L^\infty_t} &\les \|F\|_{L^{6/5, 2}_x L^\infty_t};\\
\Big\|\int_0^t S(t-s) F(s) \dd s\Big\|_{L^\infty_x L^2_t} &\les \|F\|_{L^{3/2, 1}_x L^2_t};\\
\Big\|\int_0^t S(t-s) F(s) \dd s\Big\|_{L^\infty_x L^1_t} &\les \|F\|_{L^{3/2, 1}_x L^1_t};\\
\|C(t) g\|_{L^{6, 2}_x L^\infty_t \cap L^\infty_x L^2_t} &\les \|g\|_{\dot H^1};\\
\|C(t) g\|_{L^\infty_x L^1_t} &\les \|g\|_{|\dl|^{-1} L^{3/2, 1}};\\
\Big\|\int_0^t C(t-s) G(s) \dd s\Big\|_{L^{6, 2}_x L^\infty_t \cap L^\infty_x L^2_t} &\les \|G\|_{L^1_t \dot H^1_x};\\
\Big\|\int_0^t C(t-s) G(s) \dd s\Big\|_{L^\infty_x L^1_t} &\les \|G\|_{L^1_t |\dl|^{-1} L^{3/2, 1}_x}.
\end{aligned}$$
\end{proposition}

\subsection{History of the problem} The study of this problem was initiated by Krieger--Schlag \cite{krisch}, who proved the existence of a codimension-one Lipschitz manifold of compactly supported radial initial data
$$
(\psi_0, \psi_1) \in H^3 \times H^2,\ \supp(\psi_0-\phi),\ \supp(\psi_1) \subset B(0, R),
$$
that lead to global-in-time solutions of the form
$$
\psi(x, t) = \phi(x, a_\infty) + v(x, t),
$$
where $|a_\infty-1| \les \|\psi_0 - \phi\|_{H^3} + \|\psi_1\|_{H^2}$, $\|v(t)\|_{L^\infty_x} \les \langle t \rangle^{-1} (\|\psi_0 - \phi\|_{H^3} + \|\psi_1\|_{H^2})$, and $v$ scatters.

One of their main results is the representation formula
\be\lb{repr}
\frac{\sin(t\sqrt H)P_c}{\sqrt H} = c_0 (\psi \otimes \psi) + \mc S(t),\ \|S(t) f\|_{L^\infty} \les t^{-1} \|f\|_{W^{1, 1}}.
\ee

More generally, the Cauchy problem for equation (\ref{wave}) was studied by Ginibre--Soffer--Velo \cite{gsv}. A number of results using alternate methods have been obtained for this equation by Kenig--Merle \cite{keme}, Duyckaerts--Merle \cite{dume}, Duyckaerts--Kenig--Merle \cite{dkm1}, \cite{dkm2}, \cite{dkm3}, Krieger--Schlag--Tataru \cite{kst}, Krieger--Nakanishi--Schlag \cite{kns1}, \cite{kns2}, for solutions of energy less than that of the soliton $\phi$ or slightly above it. The present work seeks to complement such results.

In the current paper we replace pointwise decay estimates such as (\ref{repr}) with reverse Strichartz estimates, proved in Proposition \ref{prop28}, in an approach derived from \cite{becgol} (which treats the case without eigenvalues or resonances). In order to deal with the resonance, we use the same method as in \cite{bec}, inspired by Yajima \cite{yajima_disp}.

This enables us to reduce the number of required derivatives from three in \cite{krisch} to one in the statement of Theorem \ref{main_theorem} and to prove the solution's continuous dependence on initial data. The same improvement also makes possible the centre-stable manifold result of Proposition \ref{invariant}.

It is likely that Theorem \ref{main_theorem} is optimal, as illustrated by the observation below. Proposition \ref{invariant} is not optimal from the point of view of scaling, but can be easily improved to become so.


\begin{observation} In condition \ref{cond}, consider the simplified ansatz
$$
\dot a(t) = a(t)^{5/4} \frac {4\pi}{\langle V, \partial_a \phi\rangle^2} \Big\langle \frac{\sin(t\sqrt{-\Delta})}{\sqrt{-\Delta}} \psi_1, V \partial_a \phi \Big\rangle.
$$
Note that $V \partial_a \phi = \Delta \partial_a \phi$, since $\partial_a \phi$ is a resonance. Then
$$
\int_0^\infty \Big\langle \Delta \partial_a \phi, \frac{\sin(t\sqrt{-\Delta})}{\sqrt{-\Delta}} \psi_1 \Big\rangle \dd t = \langle \Delta \partial_a \phi, (-\Delta)^{-1} \psi_1 \rangle = -\langle \partial_a \phi, \psi_1 \rangle.
$$
This suggests that $\psi_1$ needs to be taken in the Kato class $\mc K$ or in $L^{3/2, 1}$, so that this pairing is meaningful. Alternatively, the assumption that the initial data is in $\dot H^1 \times L^2$ only leads to $\dot a(t) \in L^2_t$, which is insufficient to close the loop. This shows that the conditions of Theorem \ref{main_theorem} may be optimal.
\end{observation}
\section{Proof of the results}
\subsection{Notations}
Let $R_0(\lambda): = (-\Delta - \lambda)^{-1}$ be the free resolvent corresponding to the free evolution $e^{it\Delta}$ and let $R_V(\lambda) := (-\Delta+V-\lambda)^{-1}$ be the perturbed resolvent corresponding to the perturbed evolution $e^{-itH}$. Explicitly, in three dimensions and for $\Im \lambda \geq 0$, 
\be\lb{eq_3.55}
R_0(\lambda^2)(x, y) = \frac 1 {4\pi} \frac {e^{i \lambda |x-y|}}{|x-y|}.
\ee

We denote Lorenz spaces by $L^{p, q}$, $1\leq p, q \leq \infty$ (see \cite{bergh} for their definition and properties), Sobolev spaces by $W^{s, p}$, $s \in \set R$, $1 \leq p \leq \infty$, and fix the Fourier transform to
$$
\widehat f(\eta) = (2\pi)^{-1} \int_{\set R} e^{-ix \eta} f(x) \dd x,\ f^{\vee}(x) = \int_{\set R^d} e^{i\eta x} f(\eta) \dd \eta.
$$
Also, let
\begin{list}{\labelitemi}{\leftmargin=1em}
\item[$\ast$] $\chi_A$ be the characteristic function of the set $A$;
\item[$\ast$] $\mc M$ be the space of finite-mass Borel measures on $\set R$;
\item[$\ast$] $\delta_x$ denote Dirac's measure at $x$;
\item[$\ast$] $\langle x \rangle = (1+|x|^2)^{\frac 1 2}$;
\item[$\ast$] $\B(X, Y)$ be the Banach space of bounded operators from $X$ to $Y$;
\item[$\ast$] $C$ be any constant (not always the same throughout the paper);
\item[$\ast$] $a \les b$ mean $|a| \leq C |b|$;
\item[$\ast$] $\mc S$ be the Schwartz space;
\item[$\ast$] $u \otimes v$ mean the rank-one operator $\langle \cdot, v \rangle u$.
\end{list}


\subsection{Setting up the problem}
Let us make the ansatz $\psi(x, t) = u(x, t) + \phi(x, a(t))$ and $a(0)=1$. Then $u$ satisfies the equation
$$\begin{aligned}
&\partial_t^2 u(t) + H(a(t)) u(t) = -\partial_t^2 (\phi(x, a(t))) + N(u(t), \phi(a(t))),\\
&u(0) = \psi_0 - \phi,\ \partial_t u(0) = \psi_1 - \dot a(0) \partial_a \phi,
\end{aligned}$$
where $H(a(t)) u = -\Delta u + V(a(t)) u$,
$$
V(a(t)) u = -5\phi^4(x, a(t)) = -5a(t)\phi^4(a(t)^{1/2}x, 1),
$$
and
$$
N(u, \phi) = 10 \phi^3 u^2 + 10 \phi^2 u^3 + 5 \phi u^4 + u^5.
$$

We assume that $\dot a \in L^1_t$ is small, so that $a(t)$ resides in a small neighborhood of $1$ for all $t$. The equation can then be rewritten
$$
\partial_t^2 u + H u = -\partial_t^2 (\phi(a(t))) + (V - V(a(t)) u(t) + N(u(t), \phi(a(t))),
$$
Following \cite{krisch}, letting $U: = \bpm u \\ \partial_t u \epm$, $\mc H: = \bpm 0 & 1 \\ -H & 0 \epm$, and
$$
W: = \bpm 0 \\ -\partial_t^2 (\phi(a(t))) + (V - V(a(t))) u(t) + N(u(t), \phi(a(t))) \epm,
$$
we obtain
\be\lb{eqn_mat}
\partial_t U =  \mc H U + W.
\ee
The spectrum of $\mc H$ consists of $i\R \cup \{\pm k\}$ and the eigenvectors corresponding to $\pm k(1)$ are
$$
G_\pm = (2k)^{-1/2} \bpm g \\ \pm k g \epm.
$$
The Riesz projections corresponding to $\pm k$ are, for $\ds J = \bpm 0 & 1 \\ -1 & 0 \epm$,
$$
P_\pm = \mp \langle \cdot, J G_\mp \rangle G_\pm.
$$
Applying these two projections to equation (\ref{eqn_mat}), we obtain, for $x_\pm(t) := \mp \langle U(t), J G_\mp \rangle$,
$$
\partial_t x_\pm(t) = \pm k x_\pm(t) \mp \langle W, J G_\mp \rangle.
$$
Solving this system of equations, we obtain
$$\begin{aligned}
x_\pm(t) &= e^{\pm t k} x_\pm(0) \mp \int_0^t e^{\pm (t-s) k} (2k)^{-1/2} \langle -\partial_s^2 (\phi(a(s))) + \\
&+(V - V(a(s))) u(s) + N(u(s), \phi(a(s))), g \rangle \dd s.
\end{aligned}$$
Integrating by parts,
$$\begin{aligned}
&\int_0^t e^{\pm (t-s) k} (2k)^{-1/2} \langle \partial_s^2 (\phi(x, a(s))), g \rangle \dd s = \\
&= (2k)^{-1/2} \dot a(t) \langle \partial_a \phi(a(t)), g \rangle - (2k)^{-1/2} e^{\pm t k} \dot a(0) \langle \partial_a \phi, g \rangle \pm \\
&\pm (2k)^{-1/2} \int_0^t e^{\pm (t-s) k} \dot a(s) k \langle \partial_a \phi(x, a(s)), g \rangle \dd s.
\end{aligned}$$
Taking into account the fact that $\langle \partial_a \phi, g \rangle = 0$, this leads to
\be\begin{aligned}\lb{xpm}
x_\pm(t) &= (2k)^{-1/2} e^{\pm t k} \langle k (\psi_0-\phi) \mp \psi_1, g \rangle \pm \\
&\pm (2k)^{-1/2} \dot a(t) \langle \partial_a \phi(a(t)) - a(t)^{-5/4} \partial_a \phi, g \rangle \mp\\
&\mp(2k)^{-1/2} \int_0^t e^{\pm (t-s) k} \langle \mp k \dot a(s) (\partial_a \phi(x, a(s)) - a(s)^{-5/4} \partial_a \phi) + \\
&+ (V - V(a(s))) u(s) + N(u(s), \phi(a(s))), g \rangle \dd s.
\end{aligned}\ee

The projection $P_c(1)$ on the $i\R$ component of the spectrum of $\mc H(1)$ is given by
$$\begin{aligned}
P_c \bpm f_0 \\ f_1 \epm &= \bpm f_0 \\ f_1 \epm + \frac 1 {2k} \Big\langle \bpm f_0 \\ f_1 \epm, \bpm -kg \\ -g \epm\Big\rangle \bpm g \\ k g \epm - \\
&- \frac 1 {2k} \Big\langle \bpm f_0 \\ f_1 \epm, \bpm k g \\ -g \epm\Big\rangle \bpm g \\ -kg \epm \\
&= \bpm f_0 - \langle f_0, g \rangle g \\ f_1 - \langle f_1, g \rangle g \epm.
\end{aligned}$$

With $P_c f = f - \langle f, g \rangle g$ being the projection on the continuous spectrum of $H$, the equation becomes
$$\begin{aligned}
P_c u(t) &= \cos(t\sqrt {H}) P_c (\psi_0 - \phi) + \frac{\sin(t\sqrt {H}) P_c}{\sqrt {H}} (\psi_1 - \dot a(0) \partial_a \phi) + \\
&+ \int_0^t \frac{\sin((t-s)\sqrt {H}) P_c}{\sqrt {H}} \big(-\partial_s^2 (\phi(a(s))) + (V - V(a(s))) u(s) + \\
&+N(u(s), \phi(a(s)))\big) \dd s.
\end{aligned}$$
A further computation shows that
$$\begin{aligned}
\int_0^t \frac{\sin((t-s)\sqrt {H}) P_c}{\sqrt {H}} \partial_s^2 \phi(a(s)) \dd s &= - \frac{\sin(t\sqrt {H}) P_c}{\sqrt {H}} \dot a(0) \partial_a \phi + \\
&+\int_0^t \cos((t-s)\sqrt {H}) P_c \dot a(s) \partial_a \phi(a(s)) \dd s.
\end{aligned}$$
We thus obtain
\be\begin{aligned}\lb{equ}
P_c u(t) &= \cos(t\sqrt {H}) P_c (\psi_0 - \phi) + \frac{\sin(t\sqrt {H}) P_c}{\sqrt {H}} \psi_1 + \\
&+ \int_0^t \frac{\sin((t-s)\sqrt {H}) P_c}{\sqrt {H}} \big((V - V(a(s))) u(s) + N(u(s), \phi(a(s)))\big) \dd s + \\
&+ \int_0^t \cos((t-s)\sqrt {H}) P_c \dot a(s) \partial_a \phi(a(s)) \dd s.
\end{aligned}\ee

Further note (following \cite{krisch}) that $\cos(t\sqrt H) P_c \partial_a \phi = \partial_a \phi$ and, for $a(s) \in [-1/2, 2]$,
$$
\partial_a \phi(x, a(s)) = a(s)^{-5/4} \partial_a \phi + (a(s)-1) O(\langle x \rangle^{-3}),
$$
where the difference satisfies symbol-type estimates under differentiation. Thus
$$\begin{aligned}
\int_0^t \cos((t-s)\sqrt {H}) P_c \dot a(s) \partial_a \phi(a(s)) \dd s = \int_0^t (1/a(s))^{5/4} \dot a(s) \phi \dd s + \\
+\int_0^t \cos((t-s)\sqrt {H}) P_c \dot a(s) (\partial_a \phi(a(s)) - a(s)^{-5/4}\partial_a \phi) \dd s.
\end{aligned}$$
At the same time, by Proposition \ref{prop28}, letting $\ds Q = - \frac {4\pi}{\langle V, \partial_a \phi \rangle^2} \partial_a \phi \otimes V \partial_a \phi$,
$$\begin{aligned}
\cos(t\sqrt {H}) P_c g &= C(t) P_c g + Q \int_0^t \cos(s\sqrt{-\Delta}) g \dd s,\\
\int_0^t \cos((t-s)\sqrt {H}) P_c G(s) \dd s &= \int_0^t C(t-s) P_c G(s) \dd s + \\
&+ Q \int_0^t \int_s^t \cos((\tau-s)\sqrt{-\Delta}) G(s) \dd s,
\end{aligned}$$
where $C(t)$ satisfies Strichartz estimates, and
$$\begin{aligned}
\frac{\sin(t\sqrt {H}) P_c}{\sqrt {H}} f &= S(t) P_c f + Q \int_0^t \frac{\sin(s\sqrt{-\Delta})}{\sqrt{-\Delta}} f \dd s,\\
\int_0^t \frac{\sin((t-s)\sqrt {H}) P_c}{\sqrt {H}} F(s) \dd s &=\int_0^t S(t-s) P_c F(s) \dd s + \\
&+Q \int_0^t \int_s^t \frac{\sin((\tau-s)\sqrt{-\Delta})}{\sqrt {-\Delta}} F(s) \dd\tau \dd s,
\end{aligned}$$
where $S(t)$ also satisfies Strichartz estimates.

We impose the condition that at any time $t \geq 0$ the multiples of $\partial_a \phi$ cancel each other out in (\ref{equ}). Taking the time derivative, this reduces to
\be\begin{aligned}\lb{cond}
&(1/a(t))^{5/4} \dot a(t) - \frac {4\pi}{\langle V, \partial_a \phi \rangle^2} \Big\langle \cos(t\sqrt{-\Delta}) (\psi_0 - \phi) + \frac{\sin(t\sqrt{-\Delta})}{\sqrt{-\Delta}} \psi_1 + \\
&+ \int_0^t \frac{\sin((t-s)\sqrt{-\Delta})}{\sqrt{-\Delta}} \big((V - V(a(s))) u(s) + N(u(s), \phi(a(s)))\big) \dd s + \\
&+ \int_0^t \cos((t-s)\sqrt{-\Delta}) \dot a(s) (\partial_a \phi(a(s)) - a(s)^{-5/4}\partial_a \phi) \dd s, V\partial_a \phi\Big\rangle = 0.
\end{aligned}\ee
A simpler condition arises from (\ref{cond}) at time zero: we obtain
$$
\dot a(0) = \frac {4\pi}{\langle V, \partial_a \phi \rangle^2} \langle V \partial_a \phi, u(0) \rangle.
$$
Assuming condition (\ref{cond}) holds, the equation for $u$ becomes
\be\begin{aligned}\lb{eqn_u}
u(t) &= C(t)(\psi_0 - \phi) + S(t) \psi_1 + \int_0^t S(t-s) \big((V - V(a(s))) u(s) + \\
&+ N(u, \phi(a(s)))\big) \dd s +\int_0^t C(t-s) \dot a(s) (\partial_a \phi(a(s)) - a(s)^{-5/4}\partial_a \phi) \dd s.
\end{aligned}\ee

Consider auxiliary variables $u_0$ and $a_0(t)$ and rewrite equations (\ref{xpm}) and (\ref{eqn_u}) and condition (\ref{cond}) thusly:
\be\lb{u_lin}
u = P_p u + P_c u = (2k)^{-1/2} (x_+(t) + x_-(t)) g + P_c u,
\ee
\be\begin{aligned}\lb{xpm_lin}
x_\pm(t) &= (2k)^{-1/2} e^{\pm t k} \langle k (\psi_0-\phi) \mp \psi_1, g \rangle \pm \\
&\pm (2k)^{-1/2} \dot a_0(t) \langle \partial_a \phi(a_0(t)) - a_0(t)^{-5/4}\partial_a \phi, g \rangle \mp\\
&\mp(2k)^{-1/2} \int_0^t e^{\pm (t-s) k} \langle \mp k \dot a_0(s) (\partial_a \phi(a_0(s)) - a_0(s)^{-5/4} \partial_a \phi)+ \\
&+ (V - V(a_0(s))) u_0(s) + N(u_0(s), \phi(a_0(s))), g \rangle \dd s,
\end{aligned}\ee
\be\begin{aligned}\lb{eqn_u_lin}
P_c u(t) &= C(t)(\psi_0 - \phi) + S(t) \psi_1 + \int_0^t S(t-s) \big((V - V(a_0(s))) u_0(s) + \\
& + N(u_0(s), \phi(a_0(s)))\big) \dd s +\int_0^t C(t-s) \dot a_0(s) (\partial_a \phi(a_0(s)) - a_0(s)^{-5/4} \partial_a \phi) \dd s,
\end{aligned}\ee
and
\be\begin{aligned}\lb{cond_lin}
&\dot a(t) = a_0(t)^{5/4} \frac {4\pi}{\langle V, \partial_a \phi \rangle^2} \Big\langle \cos(t\sqrt{-\Delta}) (\psi_0 - \phi) + \frac{\sin(t\sqrt{-\Delta})}{\sqrt{-\Delta}} \psi_1 + \\
&+ \int_0^t \frac{\sin((t-s)\sqrt{-\Delta})}{\sqrt{-\Delta}} \big((V - V(a_0(t))) u_0(t) + N(u_0(t), \phi(a_0(t)))\big) \dd s + \\
&+ \int_0^t \cos((t-s)\sqrt{-\Delta}) \dot a_0(s) (\partial_a \phi(a_0(s)) - a_0(s)^{-5/4} \partial_a \phi) \dd s, V \partial_a \phi\Big\rangle = 0.
\end{aligned}\ee
To this system of linear equations in $u$ and $\dot a$ we impose the further initial conditions $a(0) = a_0(0) = 1$. We then solve it by a fixed point method, whose two parts are presented in Sections \ref{stability} and \ref{contraction}.

\subsection{Stability}\lb{stability}

We first show that a ball in the complete metric space
$$
X:= \{(u, a) \mid u \in L^{6, 2}_x L^\infty_t \cap L^\infty_x L^2_t \cap L^\infty_x L^1_t,\ \dot a \in L^1 \cap L^\infty,\ a(0)=1\}
$$
with distance
$$
\|\delta u, \delta a\|_X = \|u\|_{L^{6, 2}_x L^\infty_t \cap L^\infty_x L^2_t \cap L^\infty_x L^1_t} + \|\dot a\|_{L^1 \cap L^\infty}.
$$
is stable under the action of the mapping $\Phi((u_0, a_0)):= (u, a)$, for sufficiently small $\|\psi_0-\phi\|_{|\dl|^{-1} L^{3/2, 1} \cap \dot H^1}$ and $\|\psi_1\|_{L^{3/2, 1} \cap L^2}$.

\begin{proposition}\lb{prop_stab} There exists $\epsilon_0$ such that for any $\epsilon<\epsilon_0$ and whenever $\|(u_0, a_0)\|_X < \epsilon$ and $\|\psi_0-\phi\|_{|\dl|^{-1} L^{3/2, 1} \cap \dot H^1} + \|\psi_1\|_{L^{3/2, 1} \cap L^2} < c \epsilon$ and $\langle k(\psi_0 - \phi) - \psi_1, g \rangle = 0$, there exists a unique $h:=h(u_0, a_0) \les \epsilon^2$ such that the system (\ref{u_lin}-\ref{cond_lin}), with initial data $\tilde \psi_0 = \psi_0 + hg$, $\tilde \psi_1 = \psi_1 - h k g$ admits a solution $(u, a) \in X$ with $\|(u, a) - (0, 1)\|_X < \epsilon$.
\end{proposition}

To begin with, assume that $\|(u_0, a_0) - (0, 1)\|_X < \epsilon<\epsilon_0$, with $\epsilon_0$ to be specified later such that in any case $\epsilon_0<1/2$. Then $a_0(t) \in (1/2, 3/2)$. 
Recall that under these circumstances
$$
|\partial_a \phi(a_0(t)) - a_0(t)^{-5/4} \partial_a \phi| \les O(\langle x \rangle^{-3}) |a_0(t)-1| \les O(\langle x \rangle^{-3}) \|\dot a_0\|_{L^1}
$$
and the difference satisfies symbol-type estimates. Thus
$$\begin{aligned}
&\|\dot a_0(t) (\partial_a \phi(a_0(t)) - a_0(t)^{-5/4} \partial_a \phi)\|_{L^1_t |\dl|^{-1} L^{3/2, 1}_x \cap L^1_t \dot H^1_x  \cap L^\infty_t |\dl|^{-1} L^{3/2, 1}_x \cap L^\infty_t \dot H^1_x} \les \\
&\les \|\dot a_0\|_{L^1 \cap L^\infty} \|\dot a_0\|_{L^1} \les \epsilon^2.
\end{aligned}$$

Along the same lines, observe that
$$
|\phi(a_0(t)) - \phi| \leq \Big|\int_1^{a_0(t)} |\partial_a \phi(a)| \dd a\Big| \leq \Big|\int_1^{1\pm\|\dot a_0\|_1} |\partial_a \phi(a)| \dd a\Big|
$$
pointwise and $\ds\Big\|\int_1^{1\pm\|\dot a_0\|_1} |\partial_a \phi(a)| \dd a\Big\|_{\dot H^1} \les \|\dot a_0\|_1 < 1/2$. Thus $\|\phi(a_0(t))\|_{L^{6, 2}_x L^\infty_t}$ is bounded by a constant.

Furthermore
$$
|V(a_0(t)) - V| \leq \Big|\int_1^{a_0(t)} |\partial_a V(a)| \dd a\Big| \leq \Big|\int_1^{1\pm\|\dot a_0\|_1} |\partial_a V(a)| \dd a\Big|
$$
pointwise and $\ds\Big\|\int_1^{1\pm\|\dot a_0\|_1} |\partial_a V(a)| \dd a\Big\|_{L^1 \cap L^\infty} \les \|\dot a_0\|_1$. Thus
$$
\|V(a_0(t)) - V\|_{L^1_x L^\infty_t \cap L^\infty_{x, t}} \les \|\dot a_0\|_1.
$$
Therefore
$$\begin{aligned}
\|(V(a_0(t))-V) u_0(t)\|_{L^{6/5, 2}_x L^\infty_t} &\les \|V(a_0(t)) - V\|_{L^{3/2, \infty}_x L^\infty_t} \|u_0\|_{L^{6, 2}_x L^\infty_t} \\
&\les \|\dot a_0\|_1 \|u_0\|_{L^{6, 2}_x L^\infty_t} \les \epsilon^2;\\
\|(V(a_0(t))-V) u_0(t)\|_{L^{3/2, 1}_x L^2_t} &\les \|V(a_0(t)) - V\|_{L^{3/2, 1}_x L^\infty_t \cap L^\infty_{x, t}} \|u_0\|_{L^\infty_x L^2_t} \\
&\les \|\dot a_0\|_1 \|u_0\|_{L^\infty_x L^2_t} \les \epsilon^2;\\
\|(V(a_0(t))-V) u_0(t)\|_{L^{3/2, 1}_x L^1_t} &\les \|V(a_0(t)) - V\|_{L^{3/2, 1}_x L^\infty_t \cap L^\infty_{x, t}} \|u_0\|_{L^\infty_x L^1_t} \\
&\les \|\dot a_0\|_1 \|u_0\|_{L^\infty_x L^1_t} \les \epsilon^2.
\end{aligned}$$
We obtain
$$
\|(V(a_0(t))-V) u_0(t)\|_{L^{6/5, 2}_x L^\infty_t \cap L^{3/2, 1}_x L^2_t \cap L^{3/2, 1}_x L^1_t} \les \epsilon^2.
$$

Next, we examine the nonlinear term $N(u_0(t), \phi(a_0(t)))$. Note that
$$\begin{aligned}
\|u_0^5\|_{L^{6/5, 2}_x L^\infty_t} &\les \|u_0\|_{L^{6, 2}_x L^\infty_t}^5 \les \epsilon^5,\\
\|u_0^5\|_{L^{3/2, 1}_x L^2_t} &\les \|u\|_{L^{6, 2}_x L^\infty_t}^4 \|u\|_{L^\infty_x L^2_t} \les \epsilon^5,\\
\|u_0^5\|_{L^{3/2, 1}_x L^1_t} &\les \|u\|_{L^{6, 2}_x L^\infty_t}^4 \|u\|_{L^\infty_x L^1_t} \les \epsilon^5
\end{aligned}
$$
and
$$\begin{aligned}
\|\phi(a_0(t))^3 u_0(t)^2\|_{L^{6/5, 2}_x L^\infty_t} &\les \|\phi(a_0(t))\|_{L^{6, 2}_x L^\infty_t}^3 \|u_0\|_{L^{6, 2}_x L^\infty_t}^2 \les \|u_0\|_{L^{6, 2}_x L^\infty_t}^2 \les \epsilon^2,\\
\|\phi(a_0(t))^3 u_0(t)^2\|_{L^{3/2, 1}_x L^2_t} &\les \|\phi(a_0(t))\|_{L^{6, 2}_x L^\infty_t}^3 \|u_0\|_{L^{6, 2}_x L^\infty_t} \|u_0\|_{L^\infty_x L^2_t} \\
&\les \|u_0\|_{L^{6, 2}_x L^\infty_t} \|u_0\|_{L^\infty_x L^2_t} \les \epsilon^2,\\
\|\phi(a_0(t))^3 u_0(t)^2\|_{L^{3/2, 1}_x L^1_t} &\les \|\phi(a_0(t))\|_{L^{6, 2}_x L^\infty_t}^3 \|u_0\|_{L^{6, 2}_x L^\infty_t} \|u_0\|_{L^\infty_x L^1_t} \\
&\les \|u_0\|_{L^{6, 2}_x L^\infty_t} \|u_0\|_{L^\infty_x L^1_t} \les \epsilon^2.
\end{aligned}$$
The other terms in $N(u_0(t), \phi(a_0(t)))$ can be treated in the same manner. Overall we obtain
$$
\|N(u_0(t), \phi(a_0(t)))\|_{L^{6/5, 2}_x L^\infty_t \cap L^{3/2, 1}_x L^2_t \cap L^{3/2, 1}_x L^1_t} \les \epsilon^2.
$$

We solve the equation with $\tilde \psi_0$, $\tilde \psi_1$ as initial data. Concerning the projection on the discrete spectrum, note that $k(\tilde \psi_0 - \phi) + \tilde \psi_1 = k(\psi_0 - \phi) + \psi_1$ and for any $1< p \leq \infty$
$$\begin{aligned}
\|x_-\|_{L^1 \cap L^\infty} &\les \|\psi_0 - \phi\|_{|\dl|^{-1} L^{3/2, 1} \cap \dot H^1} + \|\psi_1\|_{L^{3/2, 1} \cap L^2} + \\
&+\|\dot a_0\|_{L^1 \cap L^\infty} \|\partial_a \phi(a_0(t)) - a_0(t)^{-5/4} \partial_a \phi\|_{L^\infty_t L^p_x} +\\
&+ \|(V-V(a_0(t))) u_0(t)\|_{L^{3/2, 1}_x L^1_t} + \|N(u_0(t), \phi(a_0(t)))\|_{L^{3/2, 1}_x L^1_t} \\
&\les \|\psi_0 - \phi\|_{|\dl|^{-1} L^{3/2, 1} \cap \dot H^1} + \|\psi_1\|_{L^{3/2, 1} \cap L^2} + \epsilon^2.
\end{aligned}$$
Regarding $x_+$, note that $k(\tilde \psi_0 - \phi) - \tilde \psi_1 = 2hkg$ and rewrite (\ref{xpm_lin}) as
$$\begin{aligned}
x_+(t) &= (2k)^{-1/2} e^{t k} \Big(2h \langle k g, g \rangle - \int_0^\infty e^{-s k} \langle - k \dot a_0(s) (\partial_a \phi(a_0(s)) - a_0(s)^{-5/4} \partial_a \phi)+ \\
&+ (V - V(a_0(s))) u_0(s) + N(u_0(s), \phi(a_0(s))), g \rangle \dd s\Big) + \\
&+ (2k)^{-1/2} \dot a_0(t) \langle \partial_a \phi(a_0(t)) - a_0(t)^{-5/4}\partial_a \phi, g \rangle +\\
&+ (2k)^{-1/2} \int_t^\infty e^{(t-s) k} \langle - k \dot a_0(s) (\partial_a \phi(a_0(s)) - \partial_a \phi)+ \\
&+ (V - V(a_0(s))) u_0(s) + N(u_0(s), \phi(a_0(s))), g \rangle \dd s.
\end{aligned}$$
Note that
$$\begin{aligned}
&\Big\|(2k)^{-1/2} \dot a_0(t) \langle \partial_a \phi(a_0(t)) - a_0(t)^{-5/4}\partial_a \phi, g \rangle +\\
&+ (2k)^{-1/2} \int_t^\infty e^{(t-s) k} \langle - k \dot a_0(s) (\partial_a \phi(a_0(s)) - a_0(s)^{-5/4} \partial_a \phi)+ \\
&+ (V - V(a_0(s))) u_0(s) + N(u_0(s), \phi(a_0(s))), g \rangle \dd s \Big\|_{L^1_t \cap L^\infty_t} \les \\
& \les \|\dot a_0\|_{L^1 \cap L^\infty} \|\partial_a \phi(a_0(t)) - a_0(t)^{-5/4} \partial_a \phi\|_{L^\infty_t L^p_x} +\\
&+ \|(V-V(a_0(t))) u_0(t)\|_{L^{3/2, 1}_x L^1_t} + \|N(u_0(t), \phi(a_0(t)))\|_{L^{3/2, 1}_x L^1_t} \les \epsilon^2.
\end{aligned}$$
Thus $x_+ \in L^1 \cap L^\infty$ if and only if
$$\begin{aligned}
2kh \langle g, g \rangle &= \int_0^\infty e^{-s k} \langle - k \dot a_0(s) (\partial_a \phi(a_0(s)) - a_0(s)^{-5/4} \partial_a \phi)+ \\
&+ (V - V(a_0(s))) u_0(s) + N(u_0(s), \phi(a_0(s))), g \rangle \dd s.
\end{aligned}$$
This determines a unique value of $h=h(u_0, a_0)$ and computations along the same lines as above show that $|h| \les \epsilon^2$. Combining the estimates for $x_+$ and $x_-$ we obtain that for this unique value of $h$
$$
\|P_p u\|_{L^{6, 2}_x L^\infty_t \cap L^\infty_x L^2_t \cap L^\infty_x L^1_t} \les \|\psi_0 - \phi\|_{|\dl|^{-1} L^{3/2, 1} \cap \dot H^1} + \|\psi_1\|_{L^{3/2, 1} \cap L^2} + \epsilon^2.
$$

Concerning $P_c u$, note that $P_c(\tilde \psi_0-\phi) = P_c(\psi_0-\phi)$ and $P_c \tilde \psi_1 = \psi_1$. By the Strichartz estimates of Proposition \ref{prop28} we obtain that
$$\begin{aligned}
&\|P_c u\|_{L^{6, 2}_x L^\infty_t \cap L^\infty_x L^2_t \cap L^\infty_x L^1_t} \les \|\psi_0-\phi\|_{|\dl|^{-1} L^{3/2, 1} \cap \dot H^1} + \|\psi_1\|_{L^{3/2, 1} \cap L^2} + \\
&+ \|(V - V(a_0(t))) u_0(t) + N(u_0(t), \phi(a_0(t)))\|_{L^{6/5, 2}_x L^\infty_t \cap L^{3/2, 1}_x L^2_t \cap L^{3/2, 1}_x L^1_t} + \\
&+ \|\dot a_0(t) (\partial_a \phi(a_0(t)) - a_0(t)^{-5/4}\partial_a \phi)\|_{L^1_t |\dl|^{-1} L^{3/2, 1}_x \cap L^1_t \dot H^1_x} \\
& \les \|\psi_0 - \phi\|_{|\dl|^{-1} L^{3/2, 1} \cap \dot H^1} + \|\psi_1\|_{L^{3/2, 1} \cap L^2} + \epsilon^2.
\end{aligned}$$

Finally, from (\ref{cond_lin}) and Lemma \ref{lemma_free} we likewise obtain
$$\begin{aligned}
\|\dot a\|_{L^1} &\les \|\psi_0-\phi\|_{|\dl|^{-1} L^{3/2, 1}} + \|\psi_1\|_{L^{3/2, 1}} + \\
&+ \|(V - V(a_0(t))) u_0(t) + N(u_0(t), \phi(a_0(t)))\|_{L^{3/2, 1}_x L^1_t} + \\
&+ \|\dot a_0(t) (\partial_a \phi(a_0(t)) - a_0(t)^{-5/4}\partial_a \phi)\|_{L^1_t |\dl|^{-1} L^{3/2, 1}_x} \\
& \les \|\psi_0 - \phi\|_{|\dl|^{-1} L^{3/2, 1} \cap \dot H^1} + \|\psi_1\|_{L^{3/2, 1} \cap L^2} + \epsilon^2.
\end{aligned}$$

Putting together the previous estimates, we arrive at
$$
\|(u, a) - (0, 1)\|_X \les \|\psi_0 - \phi\|_{|\dl|^{-1} L^{3/2, 1} \cap \dot H^1} + \|\psi_1\|_{L^{3/2, 1} \cap L^2} + \epsilon^2.
$$
By making $\|\psi_0 - \phi\|_{|\dl|^{-1} L^{3/2, 1} \cap \dot H^1} + \|\psi_1\|_{L^{3/2, 1} \cap L^2}$ less than $c \epsilon$, for sufficiently small $\epsilon$ we arrive at $\|(u, a) - (0, 1)\|_X < \epsilon$.

\subsection{Contraction}\lb{contraction}

We next show that the mapping $(u_0, a_0) \mapsto (u, a)$ is a contraction within a sufficiently small sphere.
\begin{proposition}\lb{prop_cont}
Consider two bounded solutions $(u_0^1, a_0^1) \mapsto (u^1, a^1)$, $(u_0^2, a_0^2) \mapsto (u^2, a^2)$ of the system (\ref{u_lin}--\ref{cond_lin}) with initial data $(\psi_0^j + h(u_0^j, a_0^j)g , \psi_1^j - h(u_0^j, a_0^j)kg)$, $1 \leq j \leq 2$, such that
$$
\|\psi_0^j\|_{|\dl|^{-1} L^{3/2, 1} \cap \dot H^1} + \|\psi_1^j\|_{L^{3/2, 1} \cap L^2} < c \epsilon_1 < c/2,\ 1 \leq j \leq 2,
$$
and $\|(u_0^j, a_0^j) - (0, 1)\|_X < \epsilon_1<1/2$, $1 \leq j \leq 2$. Then
$$\begin{aligned}
\|(u^1, a^1) - (u^2, a^2)\|_X &\les \epsilon_1 \|(u_0^1, a_0^1) - (u_0^2, a_0^2)\|_X + \\
&+\|\psi_0^1-\psi_0^2\|_{|\dl|^{-1} L^{3/2, 1} \cap \dot H^1} + \|\psi_1^1-\psi_1^2\|_{L^{3/2, 1} \cap L^2}
\end{aligned}$$
and
$$
|h(u_0^1, a_0^1)-h(u_0^2, a_0^2)| \les \epsilon_1 \|(u_0^1, a_0^1) - (u_0^2, a_0^2)\|_X.
$$
\end{proposition}
\begin{proof}
A simple computation shows that
$$\begin{aligned}
&|(\partial_a \phi(a_0^1(t)) - a_0^1(t)^{-5/4} \partial_a \phi) - (\partial_a \phi(a_0^2(t)) - a_0^2(t)^{-5/4} \partial_a \phi)| \\
&\les |a_0^1(t) - a_0^2(t)| O(\langle x \rangle^{-3}) \les \|\dot a_0^1 - \dot a_0^2\|_{L^1} O(\langle x \rangle^{-3})
\end{aligned}$$
and the difference satisfies symbol-type estimates.

As in the previous section we arrive at
$$\begin{aligned}
&\|\dot a_0^1(t) (\partial_a \phi(a_0^1(t)) - a_0^1(t)^{-5/4} \partial_a \phi) - \\
&- \dot a_0^2(t) (\partial_a \phi(a_0^2(t)) - a_0^2(t)^{-5/4} \partial_a \phi)\|_{L^1_t |\dl|^{-1} L^{3/2, 1}_x \cap L^1_t \dot H^1_x  \cap L^\infty_t |\dl|^{-1} L^{3/2, 1}_x \cap L^\infty_t \dot H^1_x} \les \\
&\les \epsilon_1 \|(u_0^1, a_0^1) - (u_0^2, a_0^2)\|_X.
\end{aligned}$$
Likewise,
$$
|V(a_0^1(t)) - V(a_0^2(t))| \leq \Big|\int_{a_0^1(t)}^{a_0^2(t)} |\partial_a V(a)| \dd a\Big|
$$
implies
$$
\|V(a_0^1(t)) - V(a_0^2(t))\|_{L^1_x L^\infty_t \cap L^\infty_{x, t}} \les \|\dot a_0^1 - \dot a_0^2\|_{L^1}
$$
and thus
$$\begin{aligned}
\|(V(a_0^1(t))-V) u_0^1(t) - (V(a_0^2(t))-V) u_0^2(t)\|_{L^{6/5, 2}_x L^\infty_t \cap L^{3/2, 1}_x L^2_t \cap L^{3/2, 1}_x L^1_t} \les \\
\les \epsilon_1 \|(u_0^1, a_0^1) - (u_0^2, a_0^2)\|_X.
\end{aligned}$$
Furthermore,
$$
\|\phi(a_0^1(t)) - \phi(a_0^2(t))\|_{L^{6, 2}_x L^\infty_t} \les \|\dot a_0^1 - \dot a_0^2\|_{L^1}
$$
implies
$$\begin{aligned}
\|N(u_0^1(t), \phi(a_0^1(t))) - N(u_0^2(t), \phi(a_0^2(t)))\|_{L^{6/5, 2}_x L^\infty_t \cap L^{3/2, 1}_x L^2_t \cap L^{3/2, 1}_x L^1_t} \les \\
\les \epsilon_1 \|(u_0^1, a_0^1) - (u_0^2, a_0^2)\|_X.
\end{aligned}$$
Subtracting the two corresponding copies of (\ref{xpm_lin}) from one another, we obtain
$$\begin{aligned}
\|x_-^1 - x_-^2\|_{L^1 \cap L^\infty} &\les \epsilon_1 \|(u_0^1, a_0^1) - (u_0^2, a_0^2)\|_X + \\
&+ \|\psi_0^1-\psi_0^2\|_{|\dl|^{-1} L^{3/2, 1} \cap \dot H^1} + \|\psi_1^1-\psi_1^2\|_{L^{3/2, 1} \cap L^2}.
\end{aligned}$$
Taking the difference of the equations for $x_+^1(t)$ and $x_+^2(t)$, we obtain that $x_+^1(t) - x_+^2(t)$ is bounded if and only if
\be\begin{aligned}\lb{h1h2}
2k(h(u_0^1, a_0^1)-h(u_0^2, a_0^2)) \langle g, g \rangle &= \int_0^\infty e^{-s k} \langle - k \dot a_0^1(s) (\partial_a \phi(a_0^1(s)) - a_0^1(s)^{-5/4} \partial_a \phi)+ \\
&+ (V - V(a_0^1(s))) u_0^1(s) + N(u_0^1(s), \phi(a_0^1(s))), g \rangle \dd s - \\
&- \int_0^\infty e^{-s k} \langle - k \dot a_0^2(s) (\partial_a \phi(a_0^2(s)) - a_0^2(s)^{-5/4} \partial_a \phi)+ \\
&+ (V - V(a_0^2(s))) u_0^2(s) + N(u_0^2(s), \phi(a_0^2(s))), g \rangle \dd s.
\end{aligned}\ee
On the other hand, since $x_+^1(t)$ and $x_+^2(t)$ are in fact bounded, their difference must be bounded also, so (\ref{h1h2}) must hold. Hence
$$
|h(u_0^1, a_0^1)-h(u_0^2, a_0^2)| \les \epsilon_1 \|(u_0^1, a_0^1) - (u_0^2, a_0^2)\|_X,
$$
which was to be shown. Furthermore, under this condition
$$
\|x_+^1 - x_+^2\|_{L^1 \cap L^\infty} \les \epsilon_1 \|(u_0^1, a_0^1) - (u_0^2, a_0^2)\|_X.
$$
Consequently
$$\begin{aligned}
\|P_p u^1 - P_p u^2\|_{L^{6, 2}_x L^\infty_t \cap L^\infty_x L^2_t \cap L^\infty_x L^1_t} &\les \epsilon_1 \|(u_0^1, a_0^1) - (u_0^2, a_0^2)\|_X + \\
&+ \|\psi_0^1-\psi_0^2\|_{|\dl|^{-1} L^{3/2, 1} \cap \dot H^1} + \|\psi_1^1-\psi_1^2\|_{L^{3/2, 1} \cap L^2}.
\end{aligned}$$
By subtracting two copies of (\ref{u_lin}) from one another we likewise obtain
$$\begin{aligned}
\|P_c u^1 - P_c u^2\|_{L^{6, 2}_x L^\infty_t \cap L^\infty_x L^2_t \cap L^\infty_x L^1_t} &\les \epsilon_1 \|(u_0^1, a_0^1) - (u_0^2, a_0^2)\|_X + \\
&+\|\psi_0^1-\psi_0^2\|_{|\dl|^{-1} L^{3/2, 1} \cap \dot H^1} + \|\psi_1^1-\psi_1^2\|_{L^{3/2, 1} \cap L^2}
\end{aligned}$$
and doing the same for (\ref{cond_lin}) leads to
$$\begin{aligned}
\|\dot a^1 - \dot a^2\|_{L^1 \cap L^\infty} &\les \epsilon_1 \|(u_0^1, a_0^1) - (u_0^2, a_0^2)\|_X + \\
&+\|\psi_0^1-\psi_0^2\|_{|\dl|^{-1} L^{3/2, 1} \cap \dot H^1} + \|\psi_1^1-\psi_1^2\|_{L^{3/2, 1} \cap L^2}.
\end{aligned}$$
Thus we have proved that
$$\begin{aligned}
\|(u^1, a^1) - (u^2, a^2)\|_X &\les \epsilon_1 \|(u_0^1, a_0^1) - (u_0^2, a_0^2)\|_X + \\
&+\|\psi_0^1-\psi_0^2\|_{|\dl|^{-1} L^{3/2, 1} \cap \dot H^1} + \|\psi_1^1-\psi_1^2\|_{L^{3/2, 1} \cap L^2}.
\end{aligned}$$
\end{proof}

\subsection{Proof of the main theorem}
\begin{proof}[Proof of Theorem \ref{main_theorem}]
The proof is a straightforward application of Propositions \ref{prop_stab} and \ref{prop_cont}. For sufficiently small initial data, i.e.\ $|\psi_0-\phi\|_{|\dl|^{-1}| L^{3/2, 1} \cap \dot H^1} + \|\psi_1\|_{L^{3/2, 1} \cap L^2}$ small, take
$$
\epsilon = \epsilon_1 = C(\|\psi_0-\phi\|_{|\dl|^{-1} L^{3/2, 1} \cap \dot H^1} + \|\psi_1\|_{L^{3/2, 1} \cap L^2}).
$$
For fixed initial data $(\phi_0, \phi_1)$, consider the sequence $(u^0, a^0) = (0, 1)$, $(u^n, a^n) = \Phi(u^{n-1}, a^{n-1})$ for $n \geq 1$. By induction it follows that for every $n$, by Proposition \ref{prop_stab}, $\|(u^n, a^n) - (0, 1)\|_X < \epsilon$ and $|h(u^n, a^n)| \les \epsilon^2$. By Proposition \ref{prop_cont}
$$\begin{aligned}
\|(u^n, a^n)-(u^{n-1}, a^{n-1})\|_X + |h(u^{n-1}, a^{n-1})-h(u^{n-2}, a^{n-2}| \les \\
\les \epsilon_1 \|(u^{n-1}, a^{n-1})-(u^{n-2}, a^{n-2})\|_X.
\end{aligned}$$
For sufficiently small $\epsilon_1$ it follows that the sequence $(u^n, a^n)$ converges in $X$ to some limit $(u, a)$, which by the above must fulfill $\|(u, a) - (0, 1)\|_X \leq \epsilon$. Likewise, $h(u^n, a^n)$ converges to a limit, $h(\psi_0, \psi_1) \equiv h(u, a)$, such that $|h(\psi_0, \psi_1)| \les \epsilon^2$.

By passing to the limit in (\ref{u_lin}--\ref{cond_lin}) we obtain that $u$ and $a$ fulfill the nonlinear system (\ref{xpm}), (\ref{eqn_u}), (\ref{cond}), and (\ref{u_lin}), with initial data $(\psi_0+h(\psi_0, \psi_1) g, \psi_1 - h(\psi_0, \psi_1) kg)$. Then $\psi(\psi_0, \psi_1)(t):= u(t)+\phi(a(t))$ satisfies equation (\ref{wave}) for $t \geq 0$ with the stated initial data.

As stated above, one has that $\|(u, a) - (0, 1)\|_X \leq \epsilon$ and furthermore, by interpolation, $\|u\|_{L^8_{x, t}} \leq \|u\|_{L^{6, 2}_x L^\infty_t}^{3/4} \|u\|_{L^{\infty}_x L^2_t}^{1/4} \leq \epsilon$.

We next perform a comparison between two solutions $(u^1, a^1)$ and $(u^2, a^2)$ with different initial data, $(\psi_0^j+h(\psi_0^j, \psi_1^j) g, \psi_1^j - h(\psi_0^j, \psi_1^j) kg)$, $1 \leq j \leq 2$, again using Proposition \ref{prop_cont}. We obtain
$$\begin{aligned}
\|(u^1, a^1) - (u^2, a^2)\|_X &\les \epsilon_1 \|(u^1, a^1) - (u^2, a^2)\|_X + \\
&+ \|\psi_0^1-\psi_0^2\|_{|\dl|^{-1} L^{3/2, 1} \cap \dot H^1} + \|\psi_1^1-\psi_1^2\|_{L^{3/2, 1} \cap L^2},
\end{aligned}$$
so
$$
\|(u^1, a^1) - (u^2, a^2)\|_X \les \|\psi_0^1-\psi_0^2\|_{|\dl|^{-1} L^{3/2, 1} \cap \dot H^1} + \|\psi_1^1-\psi_1^2\|_{L^{3/2, 1} \cap L^2}.
$$
We also obtain that
$$\begin{aligned}
|h(\psi_0^1, \psi_1^1)-h(\psi_0^2, \psi_1^2)| &\les \epsilon_1 \|(u^1, a^1) - (u^2, a^2)\|_X \\
&\les \epsilon_1 (\|\psi_0^1-\psi_0^2\|_{|\dl|^{-1} L^{3/2, 1} \cap \dot H^1} + \|\psi_1^1-\psi_1^2\|_{L^{3/2, 1} \cap L^2}).
\end{aligned}$$
Furthermore,
$$\begin{aligned}
\|\psi(\psi_0^1, \psi_1^1) - \psi(\psi_0^2, \psi_1^2)\|_{L^{6, 2}_x L^\infty_t} &\leq \|u^1-u^2\|_{L^{6, 2}_x L^\infty_t} + \|\phi(a^1(t))-\phi(a^2(t))\|_{L^{6, 2}_x L^\infty_t} \\
&\les \|(u^1, a^1) - (u^2, a^2)\|_X \\
&\les \|\psi_0^1-\psi_0^2\|_{|\dl|^{-1} L^{3/2, 1} \cap \dot H^1} + \|\psi_1^1-\psi_1^2\|_{L^{3/2, 1} \cap L^2}.
\end{aligned}$$
Concerning energy, let $v = \psi(\psi_0, \psi_1)(t) - \phi = u(t) + \phi(a(t))-\phi$. Clearly
$$\begin{aligned}
\|v\|_{L^{6, 2}_x L^\infty_t} &\les \|u\|_{L^{6, 2}_x L^\infty_t} + \|\phi(a(t))-\phi\|_{L^{6, 2}_x L^\infty_t} \les \\
&\|\psi_0^1-\psi_0^2\|_{|\dl|^{-1} L^{3/2, 1} \cap \dot H^1} + \|\psi_1^1-\psi_1^2\|_{L^{3/2, 1} \cap L^2},\\
\|v\|_{L^8_t[t_1, t_2] L^8_x} &\les \|u\|_{L^8_{x, t}} + \|\phi(a(t))-\phi\|_{L^8_t[t_1, t_2] L^8_t} \les \\
&\les (1 + (t_2-t_1)^{1/8}) (\|\psi_0^1-\psi_0^2\|_{|\dl|^{-1} L^{3/2, 1} \cap \dot H^1} + \|\psi_1^1-\psi_1^2\|_{L^{3/2, 1} \cap L^2}).
\end{aligned}$$
In addition, $v$ satisfies the equation
$$\begin{aligned}\lb{eq_v}
&\partial_t^2 v(t) - \Delta v = - V v(t) + N(v(t), \phi),\\
&v(0) = \psi_0 + h(\psi_0, \psi_1) g - \phi,\ \partial_t v(0) = \psi_1 - h(\psi_0, \psi_1) kg.
\end{aligned}$$
We obtain an equation akin to (\ref{equ}): starting at time $T$
\be\begin{aligned}\lb{eq_free}
v(t) &= \cos((t-T)\sqrt {-\Delta}) v(T) + \frac{\sin((t-T)\sqrt {-\Delta})}{\sqrt {-\Delta}} \partial_t v(T) + \\
&+ \int_T^t \frac{\sin((t-s)\sqrt {-\Delta})}{\sqrt {-\Delta}} \big(- V v(s) + N(v(s), \phi)\big) \dd s.
\end{aligned}\ee
Recall the classical Strichartz estimates of Keel--Tao \cite{tao}, in particular
$$
\Big\|\cos(t\sqrt{-\Delta}) f_0 + \frac{\sin(t\sqrt{-\Delta})}{\sqrt{-\Delta}} f_1\Big\|_{L^5_t L^{10}_x} \les \|f_0\|_{\dot H^1} + \|f_1\|_{L^2}.
$$
The first two terms in (\ref{eq_free}) are clearly in $L^\infty_t \dot H^1_x \cap \dot W^{1, \infty}_t L^2_x \cap L^5_t L^{10}_x$:
$$\begin{aligned}
&\Big\|\cos((t-T)\sqrt {-\Delta}) v(T) + \frac{\sin((t-T)\sqrt {-\Delta})}{\sqrt {-\Delta}} \partial_t v(T)\Big\|_{L^\infty_t \dot H^1_x \cap \dot W^{1, \infty}_t L^2_x \cap L^5_t L^{10}_x} \les \\
&\les \|v(T)\|_{\dot H^1} + \|\partial_t v(T)\|_{L^2}.
\end{aligned}$$
Also note that
$$\begin{aligned}
&\|-V v(s) + N(v(s), \phi)\|_{L^1_s[T, T+t_0] L^2_x} \les \\
&\les t_0(\|v\|_{L^{6, 2}_x L^\infty_t} + \|v\|_{L^{6, 2}_x L^\infty_t}^3) + t_0^{1/2} \|v\|_{L^8_t[T, T+t_0] L^8_x}^4 +  \|u\|_{L^5_t[T, T+t_0] L^{10}_x}^5 \\
&\les (t_0 + t_0^{1/2}) (\|\psi_0-\phi\|_{|\dl|^{-1} L^{3/2, 1} \cap \dot H^1} + \|\psi_1\|_{L^{3/2, 1} \cap L^2}) + \|v\|_{L^5_t[T, T+t_0] L^{10}_x}^5.
\end{aligned}$$
Then
$$\begin{aligned}
&\Big\|\int_{T}^t \frac{\sin((t-s)\sqrt {-\Delta})}{\sqrt {-\Delta}} \big(- V(a(s)) u(s) + \\
&+ N(u(s), \phi(a(s)))\big) \dd s\Big\|_{L^\infty_t[T, T+t_0] \dot H^1_x \cap \dot W^{1, \infty}_t[T, T+t_0] L^2_x \cap L^5_t[T, T+t_0] L^{10}_x} \les \\
&\les \|-V(a(s))) u(s) + N(u(s), \phi(a(s)))\|_{L^1_s[T, T+t_0] L^2_x} \\
&\les (t_0 + t_0^{1/2}) (\|\psi_0-\phi\|_{|\dl|^{-1} L^{3/2, 1} \cap \dot H^1} + \|\psi_1\|_{L^{3/2, 1} \cap L^2}) + \|v\|_{L^5_t[T, T+t_0] L^{10}_x}^5.
\end{aligned}$$
By a fixed point argument, for sufficiently small $t_0$, sufficiently small $\|\psi_0-\phi\|_{|\dl|^{-1} L^{3/2, 1} \cap \dot H^1} + \|\psi_1\|_{L^{3/2, 1} \cap L^2}$, and sufficiently small $\|v(T)\|_{\dot H^1} + \|\partial_t v(T)\|_{L^2}$, we obtain that $v \in L^\infty_t[T, T+t_0] \dot H^1_x \cap \dot W^{1, \infty}_t[T, T+t_0] L^2_x \cap L^5_t[T, T+t_0] L^{10}_x$.

Let
$$
\tilde E(t):= \frac 1 2 \int_{\R^3} |\dl v(t)|^2 \dd x + \int_{\R^3} (\partial_t v(t))^2 \dd x.
$$
By conservation of energy we obtain that on any interval $[0, T]$ on which $\tilde E(t) < \infty$, for any $t \in [0, T]$
$$\begin{aligned}
E(t)&:=\frac 1 2 \int_{\R^3} |\dl \phi|^2 + 2 \dl \phi \cdot \dl v(t) + |\dl v(t)|^2 \dd x + \int_{\R^3} (\partial_t v(t))^2 \dd x - \\
&- \frac 1 6 \int_{\R^3} \phi^6 + 6 \phi^5 v(t) + 15 \phi^4 (v(t))^2 + 20 \phi^3 (v(t))^3 + 15 \phi^2 (v(t))^4 + \\
&+ 6 \phi (v(t))^5 + (v(t))^6 \dd x
\end{aligned}$$
is constant in $t$.
Since
$$
\int_{\R^3} \dl \phi \cdot \dl v(t) \dd x = - \int_{\R^3} \Delta \phi v \dd x= \int_{\R^3} \phi^5 v \dd x
$$
and $\|v\|_{L^{6, 2}_x L^\infty_t} \les \|\psi_0-\phi\|_{|\dl|^{-1} L^{3/2, 1} \cap \dot H^1} + \|\psi_1\|_{L^{3/2, 1} \cap L^2}$, we obtain that for every $t \in [0, T]$
$$\begin{aligned}
\tilde E(t) &\les \tilde E(0) + (\|\psi_0-\phi\|_{|\dl|^{-1} L^{3/2, 1} \cap \dot H^1} + \|\psi_1\|_{L^{3/2, 1} \cap L^2})^2 \\
&\les (\|\psi_0-\phi\|_{|\dl|^{-1} L^{3/2, 1} \cap \dot H^1} + \|\psi_1\|_{L^{3/2, 1} \cap L^2})^2.
\end{aligned}$$
We can then bootstrap to the interval $[0, T+t_0]$ by the above argument.

Thus the energy of $v(t)$ remains bounded for all $t$ and
$$
\|v\|_{L^5_t[t_1, t_2] L^{10}_x} \les (1 + (t_2-t_1)^{1/5}) (\|\psi_0-\phi\|_{|\dl|^{-1} L^{3/2, 1} \cap \dot H^1} + \|\psi_1\|_{L^{3/2, 1} \cap L^2}).
$$
\end{proof}

\begin{proof}[Proof of Proposition \ref{sas}]
Write $v(0) = \psi_0 + h_0 g - \phi$, $\partial_t v(0) = \psi_1 - h_0 k g$, where
$$
2k h_0 \langle g, g\rangle = \langle \langle k v(0) - \partial_t v(0), g \rangle
$$
and $\langle k(\psi_0 - \phi) - \psi_1, g \rangle = 0$. Then $|h_0| \les \epsilon$, so $\|\psi_0-\phi\|_{|\dl|^{-1} L^{3/2, 1} \cap \dot H^1} + \|\psi_1\|_{L^{3/2, 1} \cap L^2} \les \epsilon$.

For sufficiently small $\epsilon$, this means that $(\tilde \psi_0 = \psi_0+h(\psi_0, \psi_1)g, \tilde \psi_1 = \psi_1-h(\psi_0, \psi_1)kg)$ are the initial data for the global solution $\psi(\psi_0, \psi_1)(t)$ described by Theorem \ref{main_theorem}. Write $\psi(\psi_0, \psi_1)(t) = \phi + \tilde v(t)$, where $\|\tilde v(t)\|_{L^{6, 2}_x L^\infty_t} \les \epsilon$ by Theorem \ref{main_theorem}.

Both $v(t)$ and $\tilde v(t)$ admit decompositions
$$\begin{aligned}
v(t) &= (2k)^{-1/2}(x_+(t) + x_-(t)) g + P_c v(t),\\
\tilde v(t) &= (2k)^{-1/2}(\tilde x_+(t) + x_-(t)) g + P_c \tilde v(t),
\end{aligned}$$
where $x_\pm$ and $\tilde x_\pm$ satisfy the equations
$$\begin{aligned}
x_\pm(t) &= (2k)^{-1/2} e^{\pm tk} \langle k v(0) \mp \partial_t v(0), g \rangle \mp \\
&\mp (2k)^{-1/2} \int_0^t e^{\pm(t-s)k} \langle N(v(s), \phi), g \rangle \dd s,\\
\tilde x_\pm(t) &= (2k)^{-1/2} e^{\pm tk} \langle k (\tilde \psi_0 - \phi) \mp \tilde \psi_1, g \rangle \mp \\
&\mp (2k)^{-1/2} \int_0^t e^{\pm(t-s)k} \langle N(\tilde v(s), \phi), g \rangle \dd s.
\end{aligned}$$
Likewise, $P_c v(t)$ and $P_c \tilde v(t)$ satisfy the equations
$$\begin{aligned}
P_c v(t) &= \cos(t\sqrt H) P_c (\psi_0-\phi) + \frac {\sin(t\sqrt H)P_c}{\sqrt H} \psi_1 + \\
&+ \int_0^t \frac {\sin((t-s)\sqrt H)P_c}{\sqrt H} N(v(s), \phi) \dd s,\\
P_c \tilde v(t) &= \cos(t\sqrt H) P_c (\psi_0-\phi) + \frac {\sin(t\sqrt H)P_c}{\sqrt H} \psi_1 + \\
&+ \int_0^t \frac {\sin((t-s)\sqrt H)P_c}{\sqrt H} N(\tilde v(s), \phi) \dd s.
\end{aligned}$$
We subtract each pair of equations from one another. Note that
$$
kv(0)+\partial_tv(0) = k(\psi_0-\psi)+\psi_1 = k(\tilde \psi_0 - \phi) + \tilde \psi_1.
$$
Consequently
$$
x_-(t) - \tilde x_-(t) = (2k)^{-1/2} \int_0^t e^{-(t-s)k} \langle N(v(s), \phi) - N(\tilde v(s), \phi), g \rangle \dd s
$$
and so
$$
\|x_- - \tilde x_-\|_{L^\infty_t} \les \epsilon \|v - \tilde v\|_{L^{6, 2}_x L^\infty_t}.
$$
Likewise,
$$\begin{aligned}
x_+(t) - \tilde x_+(t) &= (2k)^{-1/2}\Big(2k(h_0 - h(\psi_0, \psi_1))\langle g, g \rangle - \\&-\int_0^t e^{(t-s)k} \langle N(v(s), \phi) - N(\tilde v(s), \phi), g \rangle \dd s\Big).
\end{aligned}$$
Since $x_+(t)$ and $\tilde x_+(t)$ are both bounded (here is where we use the assumption that $\|\partial_t v(t)\|_{L^\infty_t L^2_x} < \infty$), so is their difference, implying by the same process as before that
$$
2k(h_0 - h(\psi_0, \psi_1))\langle g, g \rangle = \int_0^\infty e^{-sk} \langle N(v(s), \phi) - N(\tilde v(s), \phi), g \rangle \dd s.
$$
Consequently
\be\lb{hhh}
|h_0 - h(\psi_0, \psi_1)| \les \epsilon \|v - \tilde v\|_{L^{6, 2}_x L^\infty_t}.
\ee
Then
$$
x_+(t) - \tilde x_+(t) = \int_t^\infty e^{(t-s)k} \langle N(v(s), \phi) - N(\tilde v(s), \phi), g \rangle \dd s
$$
and
$$
\|x_+ - \tilde x_+\|_{L^\infty_t} \les \epsilon \|v - \tilde v\|_{L^{6, 2}_x L^\infty_t}.
$$
Finally,
$$
P_c v(t) - P_c \tilde v(t) = \int_0^t \frac {\sin((t-s)\sqrt H)P_c}{\sqrt H} (N(v(s), \phi) - N(\tilde v(s), \phi)) \dd s,
$$
so
$$
\|P_c v - P_c \tilde v\|_{L^{6, 2}_x L^\infty_t} \les \epsilon \|v - \tilde v\|_{L^{6, 2}_x L^\infty_t}.
$$
Putting all these estimates together, we obtain that
$$
\|v - \tilde v\|_{L^{6, 2}_x L^\infty_t} \les \epsilon \|v - \tilde v\|_{L^{6, 2}_x L^\infty_t}.
$$
For sufficiently small $\epsilon$, this implies that $v \equiv \tilde v$, so by (\ref{hhh}) $h_0 = h(\psi_0, \psi_1)$. This implies that $\psi \in \mc N$.
\end{proof}

\begin{proof}[Proof of Proposition \ref{invariant}]
Consider a solution to (\ref{wave}) with initial data $(\psi_0 + h(\psi_0, \psi_1) g - \phi, \psi_1 - h(\psi_0, \psi_1) kg)$ in $\tilde {\mc N}$. Theorem \ref{main_theorem} applies and we obtain $\psi(\psi_0, \psi_1)(t) = u(t) + \phi(a(t))$ for $t \geq 0$, where
$$
a(0)=1,\ \|\dot a\|_{L^1} \les \|\psi_0 - \phi\|_{\langle x \rangle^{-1} \dot H^1} + \|\psi_1\|_{\langle x \rangle^{-1} L^2}
$$
and
$$
\|u\|_{L^{6, 2}_x L^\infty_t \cap L^\infty_x L^2_t \cap L^8_{x, t} \cap L^\infty_x L^1_t \cap L^\infty_t \dot H^1_x \cap \dot W^{1, \infty}_t L^2_x} \les \|\psi_0 - \phi\|_{\langle x \rangle^{-1} \dot H^1} + \|\psi_1\|_{\langle x \rangle^{-1} L^2}.
$$
By the local well-posedness for small data theory (which applies to all solutions, not only those with initial data on $\tilde{\mc N}$), for sufficiently small $\epsilon$, the solution also exists on some negative interval $[-t_0, 0]$ and
$$\begin{aligned}
\|\dot a\|_{L^1[-t_0, 0]} \les \|\psi_0 - \phi(a_0)\|_{\langle x \rangle^{-1} \dot H^1} + \|\psi_1\|_{\langle x \rangle^{-1} L^2},\\
\|u\|_{L^{6, 2}_x L^\infty_t[-t_0, 0] \cap L^\infty_x L^2_t[-t_0, 0] \cap L^\infty_x L^1_t[t_0, 0] \cap L^\infty_t[-t_0, 0] \dot H^1_x \cap \dot W^{1, \infty}_t[-t_0,0] L^2_x} \les \\
\les \|\psi_0 - \phi(a_0)\|_{\langle x \rangle^{-1} \dot H^1} + \|\psi_1\|_{\langle x \rangle^{-1} L^2}.
\end{aligned}$$
In fact, we could even replace $\langle x \rangle^{-1} \dot H^1 \times \langle x \rangle^{-1} L^2$ by $\dot H^1\times L^2$ in the above (for the local theory only).

As before, let $v(t) := \psi(\psi_0, \psi_1)(t) - \phi(a_0) = u(t) + \phi(a(t)) - \phi(a_0)$. Then (\ref{eq_free}) shows that
\be\begin{aligned}\lb{vv}
&\Big\|v - \frac{\sin(t\sqrt{-\Delta})}{\sqrt{-\Delta}} (\psi_0 + h(\psi_0, \psi_1) g - \phi) - \\
&- \cos(t\sqrt{-\Delta}) (\psi_1 - h(\psi_0, \psi_1) kg)\Big\|_{L^\infty_t[-t_0, t_0] \dot H^1_x \cap \dot W^{1, \infty}_t[-t_0, t_0] L^2_x} \to 0
\end{aligned}\ee
as $t_0 \to 0$. Furthermore,
$$
\|v\|_{L^{6, 2}_x L^\infty_t \cap L^\infty_t[-t_0, \infty) \dot H^1_x} + \|\partial_t v\|_{L^\infty_t[-t_0, \infty) L^2_x} \les \|\psi_0 - \phi(a_0)\|_{\langle x \rangle^{-1} \dot H^1} + \|\psi_1\|_{\langle x \rangle^{-1} L^2}.
$$
and
$$\begin{aligned}
\|v\|_{L^8_t[t_1, t_2] L^8_x} \les (1+(t_2-t_1)^{1/8}) (\|\psi_0 - \phi(a_0)\|_{\langle x \rangle^{-1} \dot H^1} + \|\psi_1\|_{\langle x \rangle^{-1} L^2}),\\
\|v\|_{L^5_t[t_1, t_2] L^{10}_x} \les (1+(t_2-t_1)^{1/5}) (\|\psi_0 - \phi(a_0)\|_{\langle x \rangle^{-1} \dot H^1} + \|\psi_1\|_{\langle x \rangle^{-1} L^2}).
\end{aligned}$$

Next, let $\tilde v(t) = x v(t)$. Then $\tilde v$ fulfills the equation
$$\begin{aligned}
&\partial_{tt} \tilde v(t) - \Delta \tilde v(t) = -2\dl v(t) - xV v(t) + \tilde N(\tilde v(t), v(t), \phi),\\
&\tilde v(0) = x(\psi_0 + h(\psi_0, \psi_1) g - \phi),\ \tilde v(1) = x(\psi_1 - h(\psi_0, \psi_1) kg).
\end{aligned}$$
where $\tilde N(\tilde v, v, \phi) := 10 \phi^3 \tilde v v + 10 \phi^2 \tilde v v^2 + 5 \phi \tilde v v ^3 + \tilde v v^4$.

However,
$$
\|-2\dl v(t) - xV v(t)\|_{L^1_t[t_1, t_2] L^2_x} \les (t_2-t_1) (\|\psi_0 - \phi\|_{\langle x \rangle^{-1} \dot H^1} + \|\psi_1\|_{\langle x \rangle^{-1} L^2})
$$
and, assuming $t_2 - t_1 \leq 1$,
$$\begin{aligned}\lb{tilden}
&\|\tilde N(\tilde v(t), v(t), \phi)\|_{L^1_t[t_1, t_2] L^2_x} \les (t_2-t_1)^{1/4} \|\tilde v\|_{L^8_t[t_1, t_2]L^8_x} \|v\|_{L^8_t[t_1, t_2]L^8_x} + \\
&+ (t_2-t_1)^{3/8} \|\tilde v\|_{L^8_t[t_1, t_2]L^8_x} \|v\|_{L^8_t[t_1, t_2]L^8_x}^2 \\
&+ (t_2-t_1)^{1/2} \|\tilde v\|_{L^8_t[t_1, t_2]L^8_x} \|v\|_{L^8_t[t_1, t_2]L^8_x}^3 \\
&+\|\tilde v\|_{L^5_t[t_1, t_2]L^{10}_x} \|v\|_{L^5_t[t_1, t_2]L^{10}_x}^4 \\
&\les ((t_2-t_1)^{1/2} + 1) \cdot \\
&\cdot (\|\psi_0 - \phi\|_{\langle x \rangle^{-1} \dot H^1} + \|\psi_1\|_{\langle x \rangle^{-1} L^2}) (\|\tilde v\|_{L^8_t[t_1, t_2]L^8_x} + \|\tilde v\|_{L^5_t[t_1, t_2]L^{10}_x}). 
\end{aligned}$$
Writing $\tilde v$ using the Duhamel formula, we again obtain that
\be\begin{aligned}\lb{tildev}
\tilde v(t) &= \cos((t-t_1)\sqrt{-\Delta}) \tilde v(t_1) + \frac{\sin((t-t_1)\sqrt{\Delta})}{\sqrt{-\Delta}} \partial_t \tilde v(t_1) + \\
&+ \int_{t_1}^t \frac{\sin((t-s)\sqrt{\Delta})}{\sqrt{-\Delta}} \big(-2\dl v(t) - xV v(t)+\tilde N(\tilde v(t), v(t), \phi)\big) \dd s,
\end{aligned}\ee
so
$$\begin{aligned}
&\|\tilde v\|_{L^\infty[t_1, t_2] \dot H^1_t \cap \dot W^{1, \infty}_t[t_1, t_2] L^2_t \cap L^8_t[t_1, t_2]L^8_x\cap L^5_t[t_1, t_2]L^{10}_x} \les \\
&\les \|\tilde v(t_1)\|_{\dot H^1} + \|\partial_t \tilde v(t_1)\|_{L^2} + (t_2-t_1) (\|\psi_0 - \phi\|_{\langle x \rangle^{-1} \dot H^1} + \|\psi_1\|_{\langle x \rangle^{-1} L^2}) + \\
&+((t_2-t_1)^{1/2} + 1) (\|\psi_0 - \phi\|_{\langle x \rangle^{-1} \dot H^1} + \|\psi_1\|_{\langle x \rangle^{-1} L^2}) \cdot \\
&\cdot (\|\tilde v\|_{L^8_t[t_1, t_2]L^8_x} + \|\tilde v\|_{L^5_t[t_1, t_2]L^{10}_x}).
\end{aligned}$$
As long as $t_2-t_1 \leq 1$ and $\|\psi_0 - \phi\|_{\langle x \rangle^{-1} \dot H^1} + \|\psi_1\|_{\langle x \rangle^{-1} L^2}$ is sufficiently small, we can solve the fixed point problem regardless of the size of the initial data and of the inhomogenous terms and obtain
$$\begin{aligned}
\|\tilde v\|_{L^\infty[t_1, t_2] \dot H^1_t \cap \dot W^{1, \infty}_t[t_1, t_2] L^2_t \cap L^8_t[t_1, t_2]L^8_x\cap L^5_t[t_1, t_2]L^{10}_x} \les \|\tilde v(t_1)\|_{\dot H^1} + \|\partial_t v(t_1)\|_{L^2} + \\
+\|\psi_0 - \phi\|_{\langle x \rangle^{-1} \dot H^1} + \|\psi_1\|_{\langle x \rangle^{-1} L^2}.
\end{aligned}$$
By bootstrapping we obtain that for $t \geq -t_0$
$$
\|\tilde v\|_{L^\infty[0, t] \dot H^1_t \cap \dot W^{1, \infty}_t[0, t] L^2_t \cap L^8_t[0, t]L^8_x\cap L^5_t[0, t]L^{10}_x} \les e^t (\|\psi_0 - \phi\|_{\langle x \rangle^{-1} \dot H^1} + \|\psi_1\|_{\langle x \rangle^{-1} L^2}).
$$
Thus $(v(t), \partial_t v(t)) \in (\langle x \rangle^{-1} \dot H^1 \times \langle x \rangle^{-1} L^2$ for every $t \in [-t_0, \infty)$.

Due to (\ref{tildev}), we also obtain that
\be\begin{aligned}\lb{tildevv}
&\Big\|\tilde v - \frac{\sin(t\sqrt{-\Delta})}{\sqrt{-\Delta}} x(\psi_0 + h(\psi_0, \psi_1) g - \phi) - \\
&- \cos(t\sqrt{-\Delta}) x(\psi_1 - h(\psi_0, \psi_1) kg)\Big\|_{L^\infty_t[-t_0, t_0] \dot H^1_x \cap \dot W^{1, \infty}_t[-t_0, t_0] L^2_x} \to 0
\end{aligned}\ee
as $t_0 \to 0$.
Thus, by (\ref{vv}) and (\ref{tildevv}), for any $\delta>0$ and sufficiently small $t_0$, for every $t \in [-t_0, t_0]$
$$
\|v(t)-(\psi_0 + h(\psi_0, \psi_1) g - \phi)\|_{\langle x \rangle^{-1} \dot H^1} + \|\partial_t v(t)\|_{\langle x \rangle^{-1} L^2} < \delta.
$$
For some $t_1 \in [-t_0, t_0]$ and $(\psi_0, \psi_1) \in \tilde{\mc N}_0$, consider the solution having $(\psi(\psi_0, \psi_1)(t_1), \partial_t \psi(\psi_0, \psi_1)(t))$ as initial data. It exists globally and, for
$$
\psi(\psi_0, \psi_1)(t)=u(t)+\phi(a(t)),
$$
it has a small $\|(u, a)\|_X$ norm (being nothing but the solution we started with, time-shifted). Then one can find $h(t_1)$ and
$$
\tilde \psi_0 = \psi(\psi_0, \psi_1)(t_1) - h(t_1) g,\ \tilde \psi_1 = \partial_t \psi(\psi_0, \psi_1)(t_1) + h(t_1) k g
$$
such that $\langle k(\tilde \psi_0-\phi)-\tilde \psi_1, g \rangle = 0$. Indeed, this equation reduces to
$$
2k h(t_1) \langle g, g \rangle = \langle kv(t_1)-\partial_tv(t_1), g \rangle.
$$
Thus $|h(t_1)| \les \delta$ and $\|(\tilde \psi_0, \tilde \psi_1)\|_{\langle x \rangle^{-1} \dot H^1 \times \langle x \rangle^{-1} L^2}$ is small. Thus $(\tilde \psi_0, \tilde \psi_1) \in \tilde{\mc N}_0$ for sufficiently small $\delta$.

Furthermore, as stated above, the solution to (\ref{wave}) $\psi(t) = u(t)+\phi(a(t))$ having $(\psi(\psi_0, \psi_1)(t_1) = \tilde \psi_0+h(t_1)g, \partial_t \psi(\psi_0, \psi_1)(t_1) = \tilde \psi_1-h(t_1)kg)$ as initial data has a small $\|(u, a)\|_X$ norm. Then, by Proposition \ref{prop_stab}, it must be the case that $h(t_1) = h(\tilde \psi_0, \tilde \psi_1)$. Thus $(\psi(\psi_0, \psi_1)(t_1), \partial_t \psi(\psi_0, \psi_1)(t_1)) \in \tilde{\mc N}$.

It follows that $\tilde{\mc N}$ is locally in time invariant.

Next, we prove that $\tilde {\mc N}$ is a centre-stable manifold for (\ref{wave}) --- or more precisely that $\tilde {\mc N} - (\phi, 0)$ is a centre-stable manifold for equation (\ref{eq_v})
$$\begin{aligned}
&\partial_t^2 v(t) - \Delta v(t) + V v(t) = N(v(t), \phi),\\
&v(0) = \psi_0 + h(\psi_0, \psi_1) g - \phi,\ \partial_t v(0) = \psi_1 - h(\psi_0, \psi_1) kg,
\end{aligned}$$
where $v(t) = \psi(t) - \phi$, relative to a small $\langle x \rangle^{-1} \dot H^1 \times \langle x \rangle^{-1} L^2$ neighborhood $\mc V = \{(v_0, v_1) \mid \|(v_0, v_1)\|_{\langle x \rangle^{-1} \dot H^1 \times \langle x \rangle^{-1} L^2} < \delta_0\}$ of the origin.

We prove that $\tilde {\mc N} - (\phi, 0)$ has the three defining properties listed in Definition \ref{centr}: $\tilde{\mc N} - (\phi, 0)$ is $t$-invariant with respect to $\mc V$, $\pi^{cs}(\tilde{\mc N} - (\phi, 0))$ contains a neighborhood of $0$ in $X^c \oplus X^s$, and $(\tilde{\mc N} -(\phi, 0))\cap W^u = \{0\}$.

The $t$-invariance of $\tilde{\mc N}-(\phi, 0)$ is a consequence of the local in time invariance proved above and has the same proof. Namely, as long as $\|(v(t), \partial_t v(t)\|_{\langle x \rangle^{-1} \dot H^1 \times \langle x \rangle^{-1} L^2}$ is kept sufficiently small, one can use the local existence theory to prove that if $(v(t), \partial_t v(t)) \in \tilde{\mc N} - (\phi, 0)$ then $(v(t+\delta t), \partial_t v(t + \delta t)) \in \tilde{\mc N} - (\phi, 0)$ for all $|\delta t| \leq \delta$, with constant $\delta$. This can be continued with an arbitrary number of steps of equal size, for as long as $\|(v(t), \partial_t v(t))\|_{\langle x \rangle^{-1} \dot H^1 \times \langle x \rangle^{-1} L^2}$ is controlled; moreover, this works both forward and backward in time.

Next, note that $\pi_{cs}((\psi_0 + h(\psi_0, \psi_1)g - \phi, \psi_1 - h(\psi_0, \psi_1) kg)) = (\psi_0 - \phi, \psi_1) \in \tilde{\mc N}_0 \equiv X^c \oplus X^s$ and $|h(\psi_0, \psi_1)| \les (\|\psi_0-\phi\|_{\langle x \rangle^{-1} \dot H^1} + \|\psi_1\|_{\langle x \rangle^{-1} L^2})^2$. Thus $\pi_{cs}(\tilde{\mc N} - (\phi, 0))$ covers a whole neighborhood of zero in $X^c \oplus X^s$.

Finally, assume that $\tilde{\mc N} - (\phi, 0)$ contained an unstable solution $v$. By Definition \ref{unstable}, $v(t)$ then exists for all $t\leq 0$, $\|(v(t), \partial_t v(t))\|_{\langle x \rangle^{-1} \dot H^1 \times \langle x \rangle^{-1} L^2} < \delta_0$ for some small $\delta_0$ and all $t\leq 0$, and $v$ decays exponentially as $t \to -\infty$, meaning that there exists $C_1>0$ such that for all $t \leq 0$ $\|(v(t), \partial_t v(t))\|_{\langle x \rangle^{-1} \dot H^1 \times \langle x \rangle^{-1} L^2} \les ~e^{C_1 t}$.

Note that in fact it suffices to assume any rate of decay as $t \to -\infty$. We also assume that $v \not \equiv 0$ in order to obtain a contradiction.

Since the norm $\|(v(t), \partial_t v(t))\|_{\langle x \rangle^{-1} \dot H^1 \times \langle x \rangle^{-1} L^2}$ is controlled for all $t \leq 0$ by a small constant, we obtain proceeding step by step that $(v(t), \partial_t v(t)) \in \mc{\tilde N}-\phi$ for all $t \leq 0$.

Then due to preservation of energy, as expressed in Theorem~\ref{main_theorem}, starting at time $t \leq 0$,
$$
\|(v(0), \partial_t v(0))\|_{\dot H^1 \times L^2} \les \|(v(t), \partial_t v(t))\|_{\langle x \rangle^{-1}\dot H^1 \times \langle x \rangle^{-1}L^2}.
$$
However, as $t \to -\infty$ the latter norm goes to zero. This leads to a contradiction if $v \not \equiv 0$.

\end{proof}

\subsection{Wiener spaces}

\begin{definition}
For a Banach lattice $X$, let the space $\V_X$ consist of kernels $T(x, y, t)$ such that, for each pair $(x, y)$, $T(x, y, t)$ is a finite measure in $t$ on $\R$ and $M(T)(x, y):=\int_{\R} \dd |T(x, y, t)|$ is an $X$-bounded operator.
\end{definition}

$\V_X$ is an algebra under
$$
(T_1 \circ T_2) (x, z, t) := \int T_1(x, y, s) T_2(y, z, t-s) \dd y \dd s.
$$
Elements of $\V_X$ have Fourier transforms
$$
\widehat T(x, y, \lambda) := \int_{\R} e^{-it\lambda} \dd T(x, y, t)
$$
and, for every $\lambda \in \R$, $T_1^{\vee}(\lambda) \circ T_2^{\vee}(\lambda) = (T_1 \circ T_2)^{\vee}(\lambda)$.

The space $\V_X$ contains elements of the form $\delta_0(t) T(x, y)$, whose Fourier transform is constantly the operator $T(x, y) \in \B(X)$. In particular, rank-one operators $\delta_0(t) \phi(x) \otimes \psi(y)$ are in $\V_p$ when $\psi \in X^*$, $\phi \in X$. More generally, $f(t) T(x, y) \in \V_X$ if $f(t) \in L^1$ and $T \in \B(X)$.

Moreover, for two Banach spaces $X$ and $Y$ of functions on $\R^3$, we also define the space $\V_{X, Y}$ of kernels $T(x, y, t)$ such that $M(T)(x, y)$ is a bounded operator from $X$ to $Y$. The set of such operators forms an algebroid.

For example, note that $R_{0}(\lambda^2) \in \V_{L^{3/2, 1}, L^\infty} \cap \V_{L^1, L^{3, \infty}}$ and $\partial_{\lambda} R_{0}(\lambda^2) \in \V_{L^1, L^\infty}$. Indeed, the Fourier transform in $\lambda$ is
$$
R_{0}^\vee(t)(x, y) = (4\pi t)^{-1} \delta_{|x-y|}(t),
$$
so $\ds M(R_{0}) = \frac 1 {4\pi|x-y|}$. Clearly $\ds \frac 1 {4\pi|x-y|}$ is in $\B(L^{3/2, 1}, L^\infty) \cap \B(L^1, L^{3, \infty})$.

Likewise, $(\partial_\lambda R_{0})^\vee(t)(x, y) = (4\pi)^{-1} \delta_{|x-y|}(t)$, so $\ds M(\partial_\lambda R_0) = (4\pi)^{-1} 1 \otimes 1$, which is in $\B(L^1, L^\infty)$.

\subsection{Regular points and regular Hamiltonians}
Before examining the possible singularity at zero, we study what happens at regular points in the spectrum.

The following two properties play an important part in the study:
\begin{lemma}\lb{lemma_22} Let $T_0(\lambda) = V R_0((\lambda+i0)^2)$, i.e.\ $\widehat T_0(t)=V(x)(4\pi t)^{-1} \delta_{|x-y|}(t)$.
\begin{list}{\labelitemi}{\leftmargin=1em}
\item[C1.] $\lim_{R\to \infty} \|\chi_{|t|\geq R}(t) \widehat T_0(t)\|_{\V_{L^1} \cap \V_{L^{3/2, 1}}} =0$.
\item[C2.] For some $n \geq 1$ $\lim_{\epsilon \to 0} \|\widehat T_0^n(t+\epsilon) - \widehat T_0^n(t)\|_{\V_{L^1} \cap \V_{L^{3/2, 1}}} = 0$.
\end{list}
\end{lemma}
These properties are shown in the course of the proof of Theorem 5 in \cite{becgol}. For the reader's convenience we reproduce the proof below.
\begin{proof}[Proof of Lemma \ref{lemma_22}]
Suppose $V$ is a bounded function with compact support in a set of
diameter $D$.  It follows that for $R > 2D$
$$
\int_{\R^3} \int_{|t|\geq R} |\widehat T_0(t) f(x)| \dd x \dd t \leq \frac 1 {4\pi} \int\int_{|x-y|\geq R}\frac {|V(x)|}{|x-y|} |f(y)| \dd y \dd x \les R^{-1} \|V\|_1 \|f\|_1
$$
and property C1 is preserved by taking the limit of $V$ in $L^{3/2, 1}$.

Next, fix $p \in (1, 4/3]$ and assume that $V$ is bounded and of compact support. Then $T_0(\lambda)$, having a kernel equal in absolute value to $\ds\frac{|V(x)|}{4\pi|x-y|}$, is uniformly bounded in $\B(X, L^p)$, $\B(L^p, X)$, and $\B(L^p)$ for all $\lambda$, where $X$ is $L^1$ or $L^{3/2, 1}$.

Since $V$ is bounded and of compact support, $\widehat T_0$ also has the local and distal properties
$$
\lim_{\epsilon \to 0} \Big\|\chi_{<\epsilon}(|x-y|) \frac{V(x)}{|x-y|}\Big\|_{\B(L^1) \cap \B(L^{3/2, 1})} = 0
$$
and
$$
\lim_{R \to \infty} \Big\|\chi_{>R}(|x-y|) \frac{V(x)}{|x-y|}\Big\|_{\B(L^1) \cap \B(L^{3/2, 1})} = 0.
$$
Combined with condition C1, this implies that for any $\epsilon>0$ there exists a cutoff function $\chi$ compactly supported in $(0, \infty)$ such that
$$
\|\chi(\rho) \widehat T_0(\rho) - \widehat T_0(\rho)\|_{\V_{L^1} \cap \V_{L^{3/2, 1}}} < \epsilon.
$$
Thus, it suffices to show that condition C2 holds for $\chi(\rho) \widehat T_0(\rho)$, where $\chi$ is a compactly supported cutoff function in $(0, \infty)$.

The Fourier transform of $\chi(\rho) \widehat T_0(\rho)$ has the form
\be\lb{ft}
(\chi(\rho) \widehat T_0(\rho))^{\vee}(\lambda) = V(x) \frac {e^{i\lambda|x-y|}}{4\pi|x-y|} \chi(|x-y|).
\ee
Such oscillating kernels have decay in the $L^p$ operator norm for $p>1$. By the Lemma of \cite{stein}, page 392,
$$
\|(\chi(\rho) \widehat T_0(\rho))^{\vee}(\lambda) f\|_{L^p} \les \lambda^{-3/p'} \|f\|_{L^p}.
$$
Therefore
$$
\|\big((\chi(\rho) \widehat T_0(\rho))^{\vee}(\lambda)\big)^N f\|_{X} \les \lambda^{-3(N-2)/p'} \|f\|_{X}.
$$
For $N>2+2p'/3$, this shows that $\partial_{\rho} (\chi(\rho) \widehat T(\rho))^N$ are uniformly bounded operators in $\B(X)$, where $X$ is either $L^1$ or $L^{3/2, 1}$. Since $(\chi(\rho) \widehat T(\rho))^N$ has compact support in $\rho$, this in turn implies C2.

For general $V \in L^{3/2, 1}$, choose a sequence of bounded compactly supported approximations for which C2 holds, as shown above. By a limiting process, we obtain that C2 also holds for $V$.
\end{proof}

\begin{lemma}\lb{lemma23}
Let $T(\lambda) = I + V R_0((\lambda+i0)^2)$. Assume that $V \in L^{3/2, 1}$ and let $\lambda_0 \ne 0$. Consider a cutoff function $\chi$. Then, for $\epsilon <<1$, $(\chi((\lambda-\lambda_0)/\epsilon) T(\lambda)^{-1})^{\wedge} \in \V_{L^1} \cap \V_{L^{3/2, 1}}$.

Likewise, infinity is a regular point: for $R >> 1$ $((1-\chi(\lambda/R)) T(\lambda)^{-1})^\wedge \in \V_{L^1} \cap \V_{L^{3/2, 1}}$.
\end{lemma}
\begin{proof}[Proof of Lemma \ref{lemma23}]
Let $S_\epsilon(\lambda) = \chi(\lambda/\epsilon)(V R_0((\lambda+i0)^2) - VR_0((\lambda_0+i0)^2))$. A simple argument based on condition C1 shows that $\lim_{\epsilon \to 0} \|\widehat S_\epsilon\|_{\V_{L^1} \cap \V_{L^{3/2, 1}}}=0$. Then, for $\epsilon < \epsilon_0/2$,
$$\begin{aligned}
\chi(\lambda/\epsilon) T^{-1}(\lambda) &= \chi(\lambda/\epsilon) \big(T(\lambda_0) + \chi(\lambda/\epsilon_0)(V R_0((\lambda+i0)^2) - V R_0((\lambda_0+i0)^2))\big)^{-1}\\
&=\chi(\lambda/\epsilon) T(\lambda_0)^{-1}(I + S_{\epsilon_0}(\lambda) \widehat T(\lambda_0)^{-1})^{-1}\\
&=\chi(\lambda/\epsilon) T(\lambda_0)^{-1} \sum_{k=0}^\infty (-1)^k (S_{\epsilon_0}(\lambda) T(\lambda_0)^{-1})^k.
\end{aligned}$$
The series above converges for sufficiently small $\epsilon_0$, showing that $(\chi(\lambda/\epsilon) T^{-1}(\lambda))^\wedge \in \V_{L^1} \cap \V_{L^{3/2, 1}}$.

At infinity, for any real number $L$ one can express the Fourier transform of $(1 - \chi(\lambda/L)) T(\lambda)$ as
\begin{equation*}
S_L(\rho) = \big(\widehat T - L\check{\eta}(L\,\cdot\,) \ast \widehat T\big)(\rho) = 
\int_\R L \check{\eta}(L\sigma) [\widehat T(\rho) - \widehat T(\rho-\sigma)]\,d\sigma
\end{equation*}
Thanks to condition C2, the norm of the right-hand integral vanishes as $L \to \infty$.  This makes it possible to construct an inverse Fourier transform for
\begin{equation*}
(1 - \chi(\lambda/2L))\big(I + T(\lambda)\big)^{-1} 
= (1 - \chi(\lambda/2L)) 
 \sum_{k=0}^\infty (-1)^k \Big(\big(1 - \chi(\lambda/L))T(\lambda)\Big)^k
\end{equation*}
via a convergent power series expansion provided $L \ge L_1$.

If only $T^N$ satisfies condition C2 then one constructs an
inverse Fourier transform for $(1-\chi(\lambda/2L))
 (I \pm T^N(\lambda))^{-1}$ via this process and observes that
\begin{equation*}
(1 - \chi(\lambda/2L))\big(I + T(\lambda)\big)^{-1}
  = (1 - \chi(\lambda/2L))\big(I + (-T(\lambda))^N \big)^{-1}
\sum_{k=0}^{N-1}(-1)^k T^k(\lambda).
\end{equation*}
\end{proof}

We next consider the effect of singularities at zero.

\subsection{The effect of resonances}
Let
$$
Q = - \frac 1 {2\pi i} \int_{|z+1|=\delta} (V R_0(0) - z)^{-1} \dd z
$$
and $\ov Q = 1-Q$. Assuming that $H=-\Delta+V$ has only a resonance $\phi$ at zero, then
$$
Q = -V \phi \otimes \phi.
$$

The resonance $\phi$ satisfies the equation $\phi = - R_0(0) V \phi$. Since $\phi \in L^{3, \infty} \cap L^\infty$, $Q$ is bounded on $L^1$ and on $L^{3/2, 1}$, so $Q \in \W$. Moreover, $Q$ is in $\B(L^1, L^{3/2, 1})$ and in $\B(L^{3/2, 1}, L^1)$.

Note that, since
$$
e^{i\lambda|x-y|}-1 \les \min(1, \lambda |x-y|) \implies e^{i\lambda|x-y|}-1 \les \lambda^{\delta} |x-y|^\delta,
$$
one has
\be\lb{holder}
V(x) \Big(\frac {e^{i\lambda|x-y|}}{|x-y|} - \frac 1 {|x-y|}\Big) \les V(x) \lambda.
\ee
Thus,
when $V \in \langle x \rangle^{-1} L^{3/2, 1}$, $\widehat T(\lambda) = I + V R_0(\lambda)$ is Lipschitz continuous in $\B(L^1)$. This implies that, more generally, when $V \in L^{3/2, 1}$ $\widehat T(\lambda)$ is continuous in $\B(L^1)$.

Let
$$
K = (I + V R_0(0) + Q)^{-1} (I - Q).
$$
Then $K$ is the inverse of $I + VR_0(0)$ in $\B(L^1) \cap \B(L^{3/2, 1})$, in the sense that
\be\lb{inverse}
K(I + V R_0(0)) = (I+VR_0(0))K = I-Q := \ov Q.
\ee

The following lemma (Lemma 4.7 from Yajima \cite{yajima_disp}) is extremely useful in studying the singularity at zero. 
\begin{lemma}\lb{lemma_invers} Let $X = X_0+X_1$ be a direct sum decomposition of a vector space $X$. Suppose that a linear operator $L \in \B(X)$ is written in the form 
$$
L = \bpm
L_{00}&L_{01}\\
L_{10}&L_{11}.
\epm.
$$
in this decomposition and that $L^{-1}_{00}$ exists. Set 
$C = L_{11} - L_{10} L_{00}^{-1} L_{01}$. Then, $L^{-1}$ exists if and only if $C^{-1}$ exists. In this case
\be\lb{invers}
L^{-1}= \bpm L_{00}^{-1} + L_{00}^{-1} L_{01} C^{-1} L_{10} L_{00}^{-1}&-L_{00}^{-1} L_{01} C^{-1} \\ 
-C^{-1} L_{10} L_{00}^{-1} & C^{-1}
\epm.
\ee
\end{lemma}

We next determine an expansion of $T(\lambda)^{-1} = (I + V R_0((\lambda+i0)^2))^{-1}$ in a neighborhood of zero.
\begin{lemma}\lb{lemma24} Assume that $\langle x \rangle V \in L^{3/2, 1}$ and that $H=-\Delta+V$ has a resonance $\phi$ at zero.
Then for $\lambda<<1$
$$
T(\lambda)^{-1} = (I + V R_0((\lambda+i0)^2))^{-1} = \widehat L(\lambda) - \lambda^{-1} \frac{4\pi i}{|\langle V, \phi\rangle|^2} V \phi \otimes \phi,
$$
where $(\chi(\lambda/\epsilon) L(\lambda))^\wedge \in \V_{L^1} \cap \V_{L^{3/2, 1}}$ locally and
$$
T^*(\lambda)^{-1} = (I + R_0((\lambda+i0)^2)V)^{-1} = \widehat L^*(\lambda) - \lambda^{-1} \frac{4\pi i}{|\langle V, \phi\rangle|^2} \phi \otimes V\phi,
$$
where $(\chi(\lambda/\epsilon) L^*(\lambda))^\wedge \in \V_{L^{\infty}} \cap \V_{L^{3, \infty}}$ for sufficiently small $\epsilon$.

\end{lemma}
\begin{proof}[Proof of Lemma \ref{lemma24}]
We apply Lemma \ref{lemma_invers} to
$$
T(\lambda) := I + V R_0((\lambda+i0)^2) = \bpm \ov Q\widehat T(\lambda) \ov Q & \ov Q\widehat T(\lambda)Q \\ Q\widehat T(\lambda)\ov Q & Q\widehat T(\lambda)Q \epm := \bpm T_{00}(\lambda) & T_{01}(\lambda) \\ T_{10} (\lambda) & T_{11} (\lambda)\epm.
$$
Note that $T_{00}(\lambda):=\ov Q(I+VR_0(\lambda^2)) \ov Q$ is invertible in $\B(\ov Q L^1)$ for $|\lambda|<<1$, because
$$
T_{00}(0) = \ov Q\widehat T(0) \ov Q = \ov Q(I + VR_0(0))\ov Q
$$
is invertible on $\ov Q L^1$ of inverse $K$, see (\ref{inverse}), and $T_{00}(\lambda)$ is continuous in the norm of $\B(L^1)$, see (\ref{holder}) above.

Assume $T_{00}(\lambda)$ were not invertible in $\B(\ov Q L^{3/2, 1})$; then by Fredholm's alternative there should exist a solution $f$ to the equation $f + VR_0(\lambda) f = 0$ in $\ov Q L^{3/2, 1}$. However, such a solution will also be in $\ov Q L^1$, which contradicts the invertibility of $T_{00}(\lambda)$ in $\B(\ov Q L^1)$.

Furthermore, start from $R_{0}((\lambda+i0)^2) \in \V_{L^{3/2, 1}, L^\infty} \cap \V_{L^1, L^{3, \infty}}$. We know that $V \in \V_{L^{3, \infty}, L^1} \cap \V_{L^{\infty}, L^{3/2, 1}}$. Thus $V R_0((\lambda+i0)^2) \in \V_{L^1} \cap \V_{L^{3/2, 1}}$ and $\ov Q$ preserves that. Then $T_{00}(\lambda) \in \V_{L^1} \cap \V_{L^{3/2, 1}}$ as well.

Next, since $T_{00}(0)$ is invertible, $T_{00}^{-1}(\lambda) \in \W_{loc}$. The proof is as follows: let $S_\epsilon(\lambda) = \chi(\lambda/\epsilon)\ov Q(VR_0((\lambda+i0)^2) - VR_0(0)) \ov Q$. A simple argument based on condition C1 shows that $\lim_{\epsilon \to 0} \|\widehat S_\epsilon\|_{\V_{L^1}\cap \V_{L^{3/2, 1}}}=0$. Then, for $\epsilon < \epsilon_0/2$,
$$\begin{aligned}
\chi(\lambda/\epsilon) T_{00}^{-1}(\lambda) &= \chi(\lambda/\epsilon) \big(T_{00}(0) + \chi(\lambda/\epsilon_0) \ov Q (V R_0((\lambda+i0)^2) - V R_0(0)) \ov Q\big)^{-1}\\
&=\chi(\lambda/\epsilon) T_{00}(0)^{-1}(I + S_{\epsilon_0}(\lambda) T_{00}(0)^{-1})^{-1}\\
&=\chi(\lambda/\epsilon) T_{00}(0)^{-1} \sum_{k=0}^\infty (-1)^k (S_{\epsilon_0}(\lambda) T_{00}(0)^{-1})^k.
\end{aligned}$$
The series above converges for sufficiently small $\epsilon_0$, showing that $\chi(\lambda/\epsilon) T_{00}^{-1}(\lambda) \in \V_{L^1} \cap \V_{L^{3/2, 1}}$.

Concerning the derivative, for $\epsilon < \epsilon_0/2$
$$
\chi(\lambda/\epsilon) \partial_\lambda T_{00}^{-1}(\lambda) = - \chi(\lambda/\epsilon) T_{00}^{-1}(\lambda) \chi(\lambda/\epsilon_0) \partial_{\lambda} T_{00}(\lambda) \chi(\lambda/\epsilon_0) T_{00}^{-1}(\lambda).
$$
In this expression $\chi(\lambda/\epsilon) T_{00}^{-1}(\lambda) \in \V_{L^1} \cap \V_{L^{3/2, 1}}$ and $\chi(\lambda/\epsilon_0) \partial_{\lambda} T_{00}(\lambda) \in \V_{L^{1}, L^{3/2, 1}}$ since $\ds M(\partial_{\lambda} T_{00}(\lambda)) = \frac{|V| \otimes 1}{4\pi}$. Thus
$\chi(\lambda/\epsilon) \partial_\lambda T_{00}^{-1}(\lambda) \in \V_{L^{1}, L^{3/2, 1}}$.

Let
$$\begin{aligned}
J (\lambda) &:= \frac{T(\lambda) - (I + V R_0(0) + i \lambda (4\pi)^{-1} V \otimes 1 )}{\lambda^2} \\
&= \frac{V R_0((\lambda+i0)^2) - V R_0(0) - i \lambda (4\pi)^{-1} V \otimes 1}{\lambda^2}.
\end{aligned}$$

Then
$$\begin{aligned}
T_{11}(\lambda) &= Q T(\lambda) Q = Q (I + V R_0(\lambda^2)) Q\\
&=Q (VR_0((\lambda+i0)^2) - V R_0(0)) Q\\
&=V \phi \otimes V\phi(R_0((\lambda+i0)^2) - R_0(0)) V\phi \otimes \phi \\
&= \Big(\lambda \frac{|\langle V, \phi \rangle|^2}{4i\pi} - \lambda^2 \langle \phi, J (\lambda) V \phi \rangle\Big) Q \\
&:=(\lambda a^{-1} - \lambda^2 \langle \phi, J(\lambda) V \phi) \rangle) Q\\
&:= \lambda c_0(\lambda) Q.
\end{aligned}$$
Note that $c_0(0) = a^{-1} \ne 0$. Here $\ds a:=\frac {4i\pi}{|\langle V, \phi\rangle|^2}$.

Here $c_0(\lambda) \in \widehat L^1$ if
$$
\int_{\R^3} \int_{\R^3} V(x)\phi(x) V(y) \phi(y) \Big\|\frac{e^{i\lambda|x-y|} - 1}{\lambda |x-y|}\Big\|_{\widehat L^1_{\lambda}} \dd x \dd y < \infty.
$$
For every $x$ and $y$,
$$
\Big\|\frac{e^{i\lambda|x-y|} - 1}{\lambda |x-y|}\Big\|_{\widehat L^1_{\lambda}}=\Big\|\frac {\chi_{[0, |x-y|]}(t)}{|x-y|}\Big\|_{L^1_t}=1,
$$
so it is enough to assume that $V \phi \in L^1$, i.e.\ that $V \in L^{3/2, 1}$, to prove that $c_0(\lambda) \in \widehat L^1$.


Regarding $J(\lambda)$, when $V \in \langle x \rangle^{-1} L^{3/2, 1}$ then
$$
\langle J(\lambda) \phi, V \phi \rangle = \Big\langle \frac{R_0((\lambda+i0)^2) - R_0(0) - i \lambda (4\pi)^{-1} 1 \otimes 1}{\lambda^2} V\phi, V \phi \Big\rangle \in \widehat L^1_\lambda.
$$

Furthermore, let
$$\begin{aligned}
\lambda\tilde\psi(\lambda) &:= \widehat T(\lambda)V \phi = (V R_0((\lambda+i0)^2)V-V R_0(0)V)\phi\\
&= \lambda\Big(i \frac{V \otimes 1}{4\pi} + \lambda J (\lambda)\Big) V\phi
\end{aligned}$$
and
$$\begin{aligned}
\lambda \tilde \psi^*(\lambda) &:= \widehat T(\lambda)^* \phi = (R_0^*((\lambda+i0)^2) V - R_0(0)V) \phi \\
&= \lambda \Big(-i \frac{1 \otimes V}{4\pi} + \lambda J^*(\lambda)\Big) \phi.
\end{aligned}$$

%
Note that $\ds M(\lambda J(\lambda)) = \frac {|V| \otimes 1}{2\pi}$. Thus $\lambda J(\lambda) \in \V_{L^1}$ for $V \in \langle x \rangle^{-1} L^{3/2, 1}$. Then $\tilde \psi(\lambda) \in \widehat L^1(L^1)$ and $\tilde \psi^*(\lambda) \in \widehat L^1(L^\infty)$.

Then
$$\begin{aligned}
T_{01}(\lambda)  &:=\ov Q \widehat T(\lambda) Q =\widehat T(\lambda) Q - Q \widehat T(\lambda) Q\\
&=-\lambda \tilde \psi(\lambda) \otimes \phi - \lambda c_0(\lambda) Q\\
&=-\lambda (\tilde \psi (\lambda) - c_0(\lambda)V\phi) \otimes \phi.
\end{aligned}$$

Likewise,
$$\begin{aligned}
T_{10}(\lambda) &= -\lambda V\phi \otimes (\tilde \psi^*(\lambda) -\ov{c_0(\lambda)}\phi).
\end{aligned}$$
Note that $T_{01}(\lambda)  = \lambda E_1(\lambda)$ with $E_1(\lambda) \in \V_{L^1} \cap \V_{L^{3/2, 1}, L^1}$ and $T_{10}(\lambda) = \lambda E_2(\lambda)$ with $E_2(\lambda) \in \V_{L^1} \cap \V_{L^1, L^{3/2, 1}}$.

Then $-T_{10} (\lambda) T_{00}^{-1}(\lambda) T_{01} (\lambda) = \lambda^2 c_1(\lambda) Q$, where
\be\lb{c1}\begin{aligned}
c_1(\lambda) &:= \big\langle \tilde \psi^*(\lambda) - \ov{c_0(\lambda)} \phi , T_{00}^{-1}(\lambda)( \tilde \psi (\lambda) - c_0 (\lambda)V\phi) \big\rangle \\
&= \Big\langle \Big(-i \frac{1 \otimes V}{4\pi} + \lambda J^*(\lambda)\Big) \phi - \ov{c_0(\lambda)}\phi,\\
&T_{00}^{-1}\Big(\Big(i \frac{V \otimes 1}{4\pi} + \lambda J (\lambda)\Big) V\phi - c_0(\lambda) V\phi\Big)\Big\rangle.
\end{aligned}\ee
For example, one of the terms in (\ref{c1}) has the form
\be\lb{term}
\langle \lambda J^*(\lambda) \phi, T_{00}^{-1}(\lambda) \lambda J(\lambda) V \phi \rangle.
\ee

Since $\lambda J(\lambda)$ and $T_{00}^{-1}(\lambda)$ are in $\V_{L^1}$ and since $\phi \in L^\infty$, $V\phi \in L^1$, it immediately follows that (\ref{term}) is in $\widehat L^1$.


We then recognize from formula (\ref{c1}) that, for a cutoff function $\chi$, $\chi(\lambda/\epsilon) c_1(\lambda) \in~\widehat L^1$.

Let
$$
C(\lambda) := T_{11}(\lambda) - T_{10}(\lambda)T_{00}^{-1}(\lambda)T_{01} (\lambda).
$$
Then
$$\begin{aligned}
C(\lambda) &= (c_0(\lambda) \lambda + \lambda ^2 c_1(\lambda)) Q \\
&= (\lambda a^{-1} - \lambda^2 \langle V_1\phi, J(\lambda) V_2\phi \rangle + \lambda^2 c_1(\lambda)) Q:= \lambda a^{-1}Q + \lambda^2 c_2(\lambda)Q.
\end{aligned}$$
Thus $C(\lambda)/\lambda$ is invertible for $|\lambda|<<1$ and 
when $V \in \langle x \rangle^{-1} L^{3/2, 1}$ one has that
$$\begin{aligned}
C^{-1}(\lambda) &= \frac 1 {\lambda a^{-1} + \lambda^2 c_2(\lambda)} Q \\
&= \Big(\frac 1 {\lambda a^{-1}} + \frac 1 {\lambda a^{-1} + \lambda^2 c_2(\lambda)} - \frac 1 {\ds\lambda a^{-1}}\Big) Q\\
&= \Big(\frac {a}{\lambda} - \frac{c_2(\lambda)}{\ds\big(a^{-1} + \lambda c_2(\lambda)\big)a^{-1}}\Big)Q\\
&:=a\lambda^{-1}Q + E(\lambda).
\end{aligned}$$
For a fixed standard cutoff function $\chi$ and sufficiently small $\epsilon$, since $Q \in \B(L^{1}) \cap \B(L^{3/2, 1}) \cap \B(L^1, L^{3/2, 1}) \cap \B(L^{3/2, 1}, L^1)$, it follows that $E(\lambda) \in \V_{L^1} \cap \V_{L^{3/2, 1}} \cap \V_{L^1, L^{3/2, 1}} \cap \V_{L^{3/2, 1}, L^1}$.

The inverse of $T$ is then given by formula (\ref{invers}):
$$
T^{-1}= \bpm T_{00}^{-1} + T_{00}^{-1} T_{01} C^{-1} T_{10} T_{00}^{-1}&-T_{00}^{-1} T_{01} C^{-1} \\ 
-C^{-1} T_{10} T_{00}^{-1} & C^{-1}
\epm.
$$
Three of the matrix elements belong to $\V_{L^1}\cap\V_{L^{3/2, 1}}$. Indeed, recall that $T_{00}^{-1}\in \W_{loc}$ and $T_{10}(\lambda) = \lambda E_1(\lambda)$ with $E_1(\lambda) \in \V_{L^1} \cap \V_{L^{3/2, 1}, L^1}$ and $T_{01}(\lambda) = \lambda E_2(\lambda)$ with $E_2(\lambda) \in \V_{L^1} \cap \V_{L^1, L^{3/2, 1}}$, while $C^{-1} = \lambda^{-1} E_3(\lambda)$, with $E_3 \in \V_{L^1} \cap \V_{L^{3/2, 1}} \cap \V_{L^1, L^{3/2, 1}} \cap \V_{L^{3/2, 1}, L^1}$ locally.

The fourth matrix element is $C^{-1}$ in the lower-right corner, which is the sum of the term $E(\lambda) \in \V_{L^1} \cap \V_{L^{3/2, 1}}$ and the singular term
$$
a \lambda^{-1} Q = -a\lambda^{-1} V\phi \otimes \phi.
$$
Thus $a \lambda^{-1}Q$ is the only singular term in the expansion of $T(\lambda)$ at zero.

The other conclusion of the theorem referring to $T^*$ is obtained by taking the adjoint.
\end{proof}

The subsequent lemma collects results obtained in \cite{becgol}.
\begin{lemma}\lb{lemma_free} The free sine and cosine evolutions satisfy the following reverse Strichartz estimates:
$$\begin{aligned}
\Big\|\frac{\sin(t\sqrt{-\Delta})}{\sqrt{-\Delta}} f\Big\|_{L^{6, 2}_x L^\infty_t \cap L^\infty_x L^2_t} &\les \|f\|_2;\\
\Big\|\frac{\sin(t\sqrt{-\Delta})}{\sqrt{-\Delta}} f\Big\|_{L^\infty_x L^1_t} &\les \|f\|_{L^{3/2, 1}};\\
\Big\|\int_0^t \frac{\sin((t-s)\sqrt{-\Delta})}{\sqrt{-\Delta}} F(s) \dd s\|_{L^{6, 2}_x L^\infty_t} &\les \|F\|_{L^{6/5, 2}_x L^\infty_t};\\
\Big\|\int_0^t \frac{\sin((t-s)\sqrt{-\Delta})}{\sqrt{-\Delta}} F(s) \dd s\|_{L^\infty_x L^2_t} &\les \|F\|_{L^{3/2, 1}_x L^2_t};\\
\Big\|\int_0^t \frac{\sin((t-s)\sqrt{-\Delta})}{\sqrt{-\Delta}} F(s) \dd s\|_{L^\infty_x L^1_t} &\les \|F\|_{L^{3/2, 1}_x L^1_t};\\
\|\cos(t\sqrt{-\Delta}) g\|_{L^{6, 2}_x L^\infty_t \cap L^\infty_x L^2_t} &\les \|g\|_{\dot H^1};\\
\|\cos(t\sqrt{-\Delta}) g\|_{L^\infty_x L^1_t} &\les \|g\|_{|\dl|^{-1} L^{3/2, 1}};\\
\Big\|\int_0^t \cos((t-s)\sqrt{-\Delta}) G(s) \dd s\|_{L^{6, 2}_x L^\infty_t \cap L^\infty_x L^2_t} &\les \|G\|_{L^1_t \dot H^1_x};\\
\Big\|\int_0^t \cos((t-s)\sqrt{-\Delta}) G(s) \dd s\|_{L^\infty_x L^1_t} &\les \|G\|_{L^1_t |\dl|^{-1} L^{3/2, 1}_x}.
\end{aligned}$$
\end{lemma}
We provide the proof, also borrowed from \cite{becgol}, for the reader's convenience.
\begin{proof}
The integral kernel of $\ds\frac{\sin(t\sqrt{-\Delta})}{\sqrt{-\Delta}}$ is given by
$$
\frac{\sin(t\sqrt{-\Delta})}{\sqrt{-\Delta}} f(x) = \frac 1{4\pi t} \int_{|x-y|=t} f(y) \dd y.
$$
Then
\be\lb{sine_evol}\begin{aligned}
\Big\|\frac {\sin(t \sqrt {-\Delta})}{\sqrt {-\Delta}} f\Big\|_{L^{\infty}_x 
L^2_t}^2 &= \esssup_x \int_0^{\infty} \Big(\int_{S^2} f(x+r\omega) r \dd \omega 
\Big)^2 \dd r\\
&\leq \esssup_x \int_0^{\infty} \Big(\int_{S^2} f(x+r\omega)^2 r^2 \dd \omega 
\Big) \Big(\int_{S^2} \dd \omega\Big) \dd r\\
&\les \|f\|_2^2.
\end{aligned}\ee
and
\be\begin{aligned}
\Big\|\frac {\sin(t \sqrt {-\Delta})}{\sqrt {-\Delta}} f\Big\|_{L^{\infty}_x 
L^1_t} &= \esssup_x \int_0^{\infty} \Big|\int_{S^2} f(x+r\omega) r \dd \omega 
\Big| \dd r\\
&\leq \esssup_x \int_\R \frac {|f(x-y)|}{|y|} \dd y\\
&\les \|f\|_{L^{3/2, 1}}.
\end{aligned}\ee
Furthermore,
$$\begin{aligned}
\Big\|\int_0^t \frac {\sin(t-s) \sqrt {-\Delta})}{\sqrt {-\Delta}} F(s) \dd s\Big\|_{L^{\infty}_x 
L^p_t} &= \esssup_x \Big\|\int_0^t \int_{|x-y|=t-s} \frac 1 {|x-y|} F(y, s) \dd y \dd s\Big\|_{L^p_t}\\
&= \esssup_x \Big\|\int_{|x-y|\leq t} \frac 1 {|x-y|} F(y, t-|x-y|) \dd y\Big\|_{L^p_t}\\
&\leq \esssup_x \int \frac 1 {|x-y|} \|F(y, t)\|_{L^p_t} \dd y\\
&\les \|F\|_{L^{3/2, 1}_x L^p_t}.
\end{aligned}$$
More generally,
$$\begin{aligned}
\Big\|\int_0^t \frac {\sin(t-s) \sqrt {-\Delta})}{\sqrt {-\Delta}} F(s) \dd s\Big\|_{ 
L^p_t}(x) \leq \int \frac 1 {|x-y|} \|F(y, t)\|_{L^p_t} \dd y.
\end{aligned}$$
Since convolution with $\ds \frac 1 {|x|}$ takes $L^{6/5, 2}$ to $L^{6, 2}$, we obtain that
$$
\frac {\sin(t\sqrt {-\Delta})}{\sqrt {-\Delta}} \in \B(L^{6/5, 2}_x L^p_t, L^{6, 2}_x L^p_t).
$$

Next, note that the integral kernel of
$$
\frac{\cos((t_1-t_2)\sqrt {{-\Delta}})}{\Delta} = ({-\Delta})^{-1} - 
\int_0^{t_1-t_2} \frac {\sin(s\sqrt{{-\Delta}})}{\sqrt{{-\Delta}}} \dd s = 
\int_{t_1 - t_2}^{\infty} \frac {\sin(s\sqrt{{-\Delta}})}{\sqrt{{-\Delta}}} \dd 
s
$$
is $T(t_1-t_2, x, y)$, where
$$
T(s, x, y) = \left\{\begin{aligned}&\frac 1 {4\pi} \frac 1 {|x-y|},&&|x-y| > 
|s|\\
&0,&& |x-y| < |s|.
\end{aligned}\right.
$$
The $L^\infty_s$ norm of this kernel is exactly $\frac 1 {4\pi} \frac 1 {|x-
y|}$, so it is a bounded operator and in particular
\be\lb{cos_bounded}\begin{aligned}
\frac{\cos((t_1-t_2)\sqrt {{-\Delta}})}{\Delta} &\in \B(L^{6/5, 2}_x L^1_t, L^{6, 2}_x L^\infty_t).
\end{aligned}\ee
Consider the operator $\displaystyle T f = \frac {\sin(t\sqrt {-\Delta})}{\sqrt H} f$. Then $\displaystyle T^* F = \int_{\R} 
\frac {\sin(t\sqrt {-\Delta})}{\sqrt {-\Delta}} F(t) \dd t$. Consequently
$$\begin{aligned}
(T T^* F)(t) &= \int_{\R} \frac {\sin(t\sqrt {-\Delta})}{\sqrt {-\Delta}} \frac {\sin(s\sqrt 
{-\Delta}) P_c}{\sqrt {-\Delta}} F(s)\dd s \\
&=  \frac 1 2 \int_{\R} \Big(\frac {\cos((t-s)\sqrt {-\Delta}) P_c}{-\Delta} - \frac 
{\cos((t+s)\sqrt {-\Delta}) P_c}{-\Delta}\Big) F(s)\dd s.
\end{aligned}$$
Thus $T T^*$ is a bounded operator from $L^{6/5, 2}_x L^1_t$ to its dual, so
\be\lb{tt1}
\frac {\sin(t\sqrt H) P_c}{\sqrt H} \in \B(L^2, L^{6, 2}_x L^{\infty}_t).
\ee
Consider $\displaystyle T f = \frac {\cos(t\sqrt {-\Delta})}{\sqrt {-\Delta}} f$. Again by the $T T^*$ method, since
$$
\int_R \frac {\cos(t\sqrt {-\Delta})}{\sqrt {-\Delta}} \frac {\cos(s\sqrt {-\Delta})}{\sqrt {-\Delta}} \dd s = \frac 
1 2 \int_{\R} \frac {\cos((t-s)\sqrt {-\Delta})}{-\Delta} + \frac {\cos((t+s)\sqrt {-\Delta}) 
P_c}{-\Delta} \dd s,
$$
we also obtain
\be\lb{tt2}
\frac {\cos(t\sqrt {-\Delta})}{\sqrt {-\Delta}} \in \B(L^2, L^{6, 2}_x L^{\infty}_t).
\ee

Likewise
$$\begin{aligned}
\big\|\cos(t \sqrt {-\Delta}) g\big\|_{L^{\infty}_x L^2_t}^2 &= \esssup_x \int_0^{\infty} \Big(\int_{S^2} f(x+r\omega) + r \partial_r f(x+r\omega) \dd \omega
\Big)^2 \dd r\\
&\les \esssup_x \int_0^{\infty} \Big(\int_{S^2} g(x+r\omega)\dd \omega\Big)^2 + \Big(\int_{S^2} \partial_r f(x+r\omega) r \dd \omega
\Big)^2 \dd r \\
&\leq \esssup_x \int_0^{\infty} \Big(\int_{S^2} g(x+r\omega)^2 \dd \omega\Big) \Big(\int_{S^2} \dd \omega\Big) + \\
&+ \Big(\int_{S^2} (\partial_r g(x+r\omega))^2 r^2 \dd \omega 
\Big) \Big(\int_{S^2} \dd \omega\Big) \dd r\\
&\les \|\dl g\|_2^2.
\end{aligned}
$$
Since for bounded compactly supported functions $g$
$$
\int_\R \int_{S^2} |g(x+r\omega)| \dd \omega \dd r = - \int_\R \int_{S^2} \partial_r |g(x+r\omega)| r \dd \omega \dd r \les \|\dl g\|_{L^{3/2, 1}},
$$
by approximation we obtain that for all $g \in |\dl|^{-1}| L^{3/2, 1} \subset L^{3, 1}$
$$\begin{aligned}
\big\|\cos(t \sqrt {-\Delta}) g\big\|_{L^{\infty}_x L^1_t} &= \esssup_x \int_0^{\infty} \Big|\int_{S^2} g(x+r\omega) + r \partial_r g(x+r\omega) \dd \omega \Big| \dd r\\
&\leq \esssup_x \int_0^{\infty} \int_{S^2} |g(x+r\omega)| \dd \omega \dd r + \int_0^\infty \int_{S^2} |\partial_r g(x+r\omega)| r \dd \omega \dd r\\
&\les \|\dl g\|_{L^{3/2, 1}}.
\end{aligned}
$$
The last remaining result follows by Minkowski's inequality.
\end{proof}

FInally, we can prove Proposition \ref{prop28}.
\begin{proof}[Proof of Proposition \ref{prop28}]
Recall that by Lemma \ref{lemma24} for $\lambda<<1$
$$
\widehat T^*(\lambda)^{-1} = (I + R_0(\lambda^2)V)^{-1} = \widehat L^*(\lambda) - \lambda^{-1} \frac{4\pi i}{|\langle V, \phi\rangle|^2} \phi \otimes V\phi,
$$
where $(\chi(\lambda/\epsilon) L^*)^\wedge \in \V_{L^{\infty}} \cap \V_{L^{3, \infty}}$ for sufficiently small $\epsilon$.

We consider a partition of unity subordinated to the neighborhoods of Lemmas \ref{lemma23} and \ref{lemma24}:
$$
1 = \chi_0(\lambda/\epsilon) + \sum_{k=1}^N \chi_k((\lambda-\lambda_k)/\epsilon_k) + (1-\chi_{\infty}(\lambda/R)).
$$
By Lemma \ref{lemma23}, for any $\lambda_k \ne 0$, $(\chi_k((\lambda-\lambda_k)/\epsilon_k) T^*(\lambda))^\wedge \in \V_{L^{3, \infty}} \cap \V_{L^\infty}$ and same at infinity. The sum of these terms has the same property. By Lemma \ref{lemma24} $\chi_0(\lambda/\epsilon) \widehat T(\lambda)$ also decomposes into $(\chi_0(\lambda/\epsilon) L^*)^\wedge \in \V_{L^{3, \infty}} \cap \V_{L^\infty}$ and the singular term
$$
-\chi_0(\lambda/\epsilon)\lambda^{-1} \frac{4\pi i}{|\langle V, \phi\rangle|^2} \phi \otimes V\phi.
$$
Write
$$\begin{aligned}
-\chi_0(\lambda/\epsilon)\lambda^{-1} \frac{4\pi i}{|\langle V, \phi\rangle|^2} \phi \otimes V\phi &= -\lambda^{-1} \frac{4\pi i}{|\langle V, \phi\rangle|^2} \phi \otimes V\phi +\\
&+(1-\chi_0(\lambda/\epsilon))\lambda^{-1} \frac{4\pi i}{|\langle V, \phi\rangle|^2} \phi \otimes V\phi.
\end{aligned}$$
The Fourier transform of the second term is given by
\be\lb{scalar}
-\frac{2\pi}{|\langle V, \phi\rangle|^2} (\sgn(t)-(\chi_0(\cdot/\epsilon))^\wedge(t)\ast\sgn(t)) \phi \otimes V\phi.
\ee
When $\chi_0(\lambda)=\chi_0(-\lambda)$, the scalar coefficient in (\ref{scalar}) is bounded in absolute value by $\ds \frac{4\pi}{|\langle V, \phi\rangle|^2} \int_{|t|}^\infty |(\chi_0(\cdot/\epsilon))^\wedge(s)| \dd s$, so it is in $L^1_t$ for sufficiently smooth $\chi_0$. Thus this term belongs to $\V_{L^{3, \infty}} \cap \V_{L^\infty}$.

Let $\chi_{00}(\lambda) = \chi_0(\lambda/\epsilon)$ and $Z_1$ be given by the sum of all the $\V_{L^{3, \infty}} \cap \V_{L^\infty}$ terms of the decomposition, so
$$\begin{aligned}
Z(\lambda) &= (1-\chi_{00}(\lambda))\lambda^{-1} \frac{4\pi i}{|\langle V, \phi\rangle|^2} \phi \otimes V\phi + \chi_{00}(\lambda) L^*(\lambda) - (1-\chi_{00}(\lambda)) T^*(\lambda) \\
&= (I + R_0((\lambda+i0)^2) V)^{-1} + \lambda^{-1} \frac {4\pi i}{|\langle V, \phi \rangle|^2} \phi \otimes V\phi \\
&= I - R_V((\lambda+i0)^2) V + \lambda^{-1} \frac {4\pi i}{|\langle V, \phi \rangle|^2} \phi \otimes V\phi.
\end{aligned}$$
Thus, since $\langle g, \partial_a \phi \rangle = 0$,
\be\lb{Z}
P_c Z(\lambda) = P_c - P_c R_V((\lambda+i0)^2) V + \lambda^{-1} \frac {4\pi i}{|\langle V, \phi \rangle|^2} \phi \otimes V\phi.
\ee
Note that $I - R_V(\lambda) V$ is meromorphic on $\C \setminus [0, \infty)$ with only one pole at $-k^2$. Let $Q_{-k^2}$ be the spectral projection corresponding to $-k^2$, given by
$$
Q_{-k^2} = \lim_{\epsilon \to 0} \frac 1 {2\pi i} \int_{|\theta| = \epsilon} I - R_V(-k^2+\theta)V \dd \theta.
$$
Clearly $Q_{-k^2} = -g \otimes gV$ and $Q_{-k^2} R_0(-k^2) = g \otimes g = P_p = I-P_c$.

Then $P_c - R_V(\lambda^2) V - (\lambda^2+k^2)^{-1} Q_{-k^2}$ has a weakly analytic continuation to the upper half-plane, which is weakly continuous on $\R$ except at zero.

Note that $R_V(\lambda^2) g \otimes g = (\lambda^2+k^2)^{-1} g \otimes g$, so $P_c - R_V(\lambda^2) V - (\lambda^2+k^2)^{-1} Q_{-k^2} = P_c - P_c R_V(\lambda^2) V)$. 

Furthermore, for $\Im \lambda > 0$ and $f \in L^2$
$$\begin{aligned}
&\int_0^\infty e^{it\lambda}  \frac{\sin(t\sqrt H) P_c}{\sqrt H} f \dd t = \\
&= \int_0^\infty e^{it\lambda} \frac 1 {2\pi i} \int_0^\infty \sin(t\eta) (R_V(\eta^2+i0) - R_V(\eta^2-i0)) f 2\dd \eta \dd t.
\end{aligned}$$
Since the integral is absolutely convergent, we may interchange the order of integration. Because
$$
\int_0^\infty e^{it\lambda} \sin(t\eta) \dd t = \frac 1 {2i} \int_0^\infty e^{it(\lambda+\eta)} - e^{it(\lambda-\eta)} \dd t = \frac{\eta}{\eta^2-\lambda^2},
$$
we obtain that
$$
\int_0^\infty e^{it\lambda} \frac{\sin{t\sqrt H} P_c}{\sqrt H} \dd t = P_c R_V(\lambda^2).
$$
Likewise
$$
\int_0^\infty e^{it\lambda} \frac{\sin{t\sqrt {-\Delta}}}{\sqrt {-\Delta}} \dd t = R_0(\lambda^2).
$$

Since $P_c R_V(\lambda^2) = P_c(I-R_V(\lambda^2)V)R_0(\lambda^2)$ for $\Im \lambda \geq0$, the inverse Fourier transforms of $\ds \chi_{t \geq 0}(t)\frac{\sin(t\sqrt H)P_c}{\sqrt H}$ and
$$
\int_{-\infty}^t (P_c(I-R_V(\lambda^2)V))^\wedge(t-s) \chi_{s \geq 0}(s) \frac{\sin(s\sqrt {-\Delta})}{\sqrt{-\Delta}} \dd s
$$
coincide for $\Im \lambda >0$.

By (\ref{Z}) $P_c Z(\lambda)$ is also weakly analytic for $\Im \lambda>0$ and weakly continuous on $\R$ except possibly at zero. In the upper half-plane, $\ds \lambda^{-1} \frac {4\pi i}{|\langle V, \phi \rangle|^2} \phi \otimes V\phi$ is the inverse Fourier transform of $\ds \frac {-4\pi}{|\langle V, \phi \rangle|^2} \chi_{[0, \infty)}(t) \phi \otimes V\phi$. Consequently, the inverse Fourier transforms of $\ds \chi_{[0, \infty)}(t)\frac{\sin(t\sqrt H)P_c}{\sqrt H}$ and
$$
\tilde S(t) = \int_{-\infty}^\infty (P_c \widehat Z(t-s) - \frac {4\pi}{|\langle V, \phi \rangle|^2} \chi_{[0, \infty)}(t-s) \phi \otimes V\phi) \chi_{s \geq 0}(s) \frac{\sin(s\sqrt {-\Delta})}{\sqrt{-\Delta}} \dd s
$$
coincide for $\Im \lambda > 0$.

Let $Z_+ = (\chi_{[0, \infty)}(t) P_c\widehat Z(t))^\vee$ and $Z_- = (\chi_{(-\infty, 0)}(t) P_c\widehat Z(t))^\vee$. Clearly $Z_+$ is weakly analytic for $\Im \lambda>0$ and weakly continuous on $\R$ and $Z_-$ is analytic for $\Im \lambda<0$ and weakly continuous on $\R$. Since $Z_- = P_c Z - Z_+$, it follows that $Z_-$ is weakly analytic for both $\Im \lambda>0$ and $\Im \lambda<0$ and weakly continuous on $\R$ possibly except at zero from above. However, this implies that $Z_-$ is weakly analytic on the whole complex plane. Since $Z_-$ is uniformly bounded in the operator norm, it follows that $Z_-$ is constant, so it must be $0$. This implies that the support of $P_c\widehat Z(t)$ is on $[0, \infty)$, so the same is true for $\tilde S(t)$.

Then note that
$$\Big(e^{-yt} \chi_{[0, \infty)}(t)\frac{\sin(t\sqrt H)P_c}{\sqrt H}\Big)^\wedge(\lambda) = \Big(\chi_{[0, \infty)}(t)\frac{\sin(t\sqrt H)P_c}{\sqrt H}\Big)^\wedge(\lambda+iy)
$$
and $(e^{-yt} \tilde S(t))^\wedge = (\tilde S(t))^\wedge(\lambda+iy)$, so the inverse Fourier transforms coincide for $y>0$. This implies that the expressions themselves coincide, so
$$
\chi_{[0, \infty)}(t)\frac{\sin(t\sqrt H)P_c}{\sqrt H} = \int_0^t (P_c \widehat Z(t-s) - \frac {4\pi}{|\langle V, \phi \rangle|^2} \phi \otimes V\phi) \frac{\sin(s\sqrt {-\Delta})}{\sqrt{-\Delta}} \dd s.
$$
Taking the derivative we obtain as well that
$$
\chi_{[0, \infty)}(t)\cos(t\sqrt H)P_c = \int_0^t \Big(P_c \widehat Z(t-s) - \frac {4\pi}{|\langle V, \phi \rangle|^2} \phi \otimes V\phi\Big) \cos(s\sqrt {-\Delta}) \dd s.
$$

Since $\widehat Z \in \V_{L^{3, \infty}} \cap \V_{L^\infty}$, it follows that
$$
\widehat Z \in \B(L^{6, 2}_x L^{\infty}_t) \cap \B(L^{\infty}_x L^2_t) \cap \B(L^{\infty}_x L^1_t).
$$
By Lemma \ref{lemma_free} $\ds \frac{\sin(t\sqrt {-\Delta})}{\sqrt{-\Delta}} \in \B(L^2_x, L^{6, 2}_x L^{\infty}_t \cap L^{\infty}_x L^2_t) \cap \B(L^{3/2, 1}_x, L^\infty_x L^1_t) \cap \B(L^{6/5, 2}_x L^{\infty}_t, L^{6, 2}_x L^{\infty}_t) \cap \B(L^{3/2, 1}_x L^2_t, L^{\infty}_x L^2_t) \cap \B(L^{3/2, 1}_x L^1_t, L^\infty_x L^1_t)$, so the conclusion follows for $\frac{\sin(t\sqrt H)P_c}{\sqrt H}$.

Likewise, $\cos(t\sqrt{-\Delta}) \in \B(\dot H^1_x, L^{6, 2}_x L^{\infty}_t \cap L^{\infty}_x L^2_t) \cap \B(|\dl|^{-1} L^{3/2, 1}_x, L^\infty_x L^1_t) \cap \B(L^1_t \dot H^1_x, L^{6, 2}_x L^{\infty}_t \cap L^{\infty}_x L^2_t) \cap \B(L^1_t |\dl|^{-1} L^{3/2, 1}_x, L^\infty_x L^1_t)$, so the conclusion also follows for $\cos(t\sqrt H)P_c$.
\end{proof}

\section*{Acknowledgments} I would like to thank Professor Wilhelm Schlag for introducing me to this problem.

\end{document}